\documentclass[12pt]{article}
{\begin{Sbox}\begin{minipage}}%
{\end{minipage}\end{Sbox}\fbox{\TheSbox}}%

\usepackage[psamsfonts]{amssymb}
\usepackage{amsmath, amsthm, anysize, enumerate, color, url}
\usepackage{graphicx}
\usepackage{epstopdf}
\usepackage{epsfig}
\usepackage{subfigure}
\usepackage[ruled,linesnumbered]{algorithm2e}
\usepackage{booktabs}

\usepackage[sectionbib]{natbib}
\usepackage{threeparttable}

\usepackage[colorlinks, breaklinks,
                   linkcolor=red,
                  citecolor=blue
                  ]{hyperref}

\usepackage{breakurl}

\newtheorem{theorem}{Theorem} 
\newtheorem{corollary}{Corollary}

\newtheorem{lemma}{Lemma} 
\newtheorem{proposition}{Proposition} 

\theoremstyle{definition}
\newtheorem{remark}{Remark}  

\newcommand{\E}{\mathbb{E}}

\newcommand{\R}{\mathbb{R}}

\renewcommand{\P}{\mathbb{P}}

\usepackage{setspace}

\begin{document}

\title{Moment estimation in paired comparison models with a growing number of subjects}

\author{
Qiuping Wang\thanks{School of Mathematics and Statistics, Zhaoqing University, Zhaoqing, 526000, China.
\texttt{Email:} qp.wang@mails.ccnu.edu.cn}
\hspace{4mm}
Lu Pan\thanks{School of Mathematics and Statistics, Shangqiu Normal University, Shangqiu, 476000, China.
\texttt{Email:} panlulu@mails.ccnu.edu.cn}
\hspace{4mm}
Ting Yan\thanks{Department of Statistics, Central China Normal University, Wuhan, 430079, China.
\texttt{Email:} tingyanty@mail.ccnu.edu.cn}
}
\date{}

\maketitle

\begin{abstract}
When the number of subjects, $n$, is large, paired comparisons are often sparse.
Here, we study statistical inference in a class of paired comparison models parameterized by a set of merit parameters, under an
Erd\"{o}s--R\'{e}nyi comparison graph, where the sparsity is measured by a probability $p_n$ tending to zero.
We use the moment estimation base on the scores of subjects to infer the merit parameters.
We establish a unified theoretical framework in which the uniform consistency and asymptotic normality of the moment estimator hold
as the number of subjects goes to infinity.
A key idea for the proof of the consistency is that we obtain the convergence rate of the Newton iterative sequence for solving the estimator.
We use the Thurstone model to illustrate the unified theoretical results.
Further extensions to a fixed sparse comparison graph are also provided.
Numerical studies and a real data analysis illustrate our theoretical findings.


\end{abstract}

\vskip 5pt

\section{Introduction}

Subjects are repeatedly compared in pairs in a wide spectrum of situations,
including sports games [\cite{han2023general}], ranking of scientific journals [\cite{stigler1994citation, Varin-2016-jrsa}], the quality of product brands [\cite{radlinski2007active}] and crowdsourcing [\cite{chen2016overcoming}]. For instances, one team plays with another team in basketballs;
papers in one journal cite papers in another journal;
one consumer choose one product over another; workers in a crowdsourcing setup are asked to compare pairs of items. 

One of fundamental problems is to give a ranking of all subjects based on the observed paired comparison data.
Since there is generally no natural ranking of subjects but for the round-robin tournaments,
paired comparison models have been developed to address this issue; see the book by \cite{David1988}.
Statistical models not only provide a method of ranking all subjects, but also are tools for making inference
on the merits of subjects (e.g., testing whether two subjects have the same merits).

Here, we are concerned with a class of paired comparison models that assign one merit parameter to each subject and
assume that the win-loss probability of any pair only depends on the difference of their merit parameters. Specifically,
the probability of subject $i$ winning $j$ is
\begin{equation}
\label{model}
\P( i \mbox{~wins~} j) = F( \beta_i - \beta_j),~~i,j=0, 1, \ldots, n; i\neq j,
\end{equation}
where $F$ is  a known cumulative distribution function satisfying $F(x)=1-F(-x)$, $\beta_i$ is the merit
parameter of subject $i$  and $n+1$ is the total number of subjects.
The well-known Bradley-Terry model [\cite{bradley-terry1952}], which dates back to at least 1929 [\cite{zermelo1929berechnung}], and the Thurstone model [\cite{thurstone1927a}] are two special cases of model \eqref{model}.
The former postulates the logistic distribution of $F(x)$ while
the latter postulates the normal distribution.

In the standard setting that $n$ is fixed and the number of comparisons in each pair goes to infinity,
the theoretical properties of the model \eqref{model} have been widely investigated in Chapter 4 of \cite{David1988}.
In the opposite scenario that $n$ goes to infinity and each pair has a fixed number of comparisons, \cite{simons-yao1999} proved the uniform
consistency and asymptotic normality of the maximum likelihood estimator (MLE) in the Bradley-Terry model.

When the number of subjects is large, paired comparisons are often sparse.
Taking the NCAA Division I FBS (Football Bowl Subdivision) regular season for example, a team plays with at most $14$ other teams among a total of $120$ teams.
The observed comparisons can be represented in a comparison graph with $n+1$ nodes denoting subjects and
a weighted edge between two nodes  denoting the number of comparisons.
The Erd\"{o}s--R\'{e}nyi comparison graph has been widely considered in literature [e.g., \cite{chen2015-confernece, Nihar2018, chen2019aoap, han2023general}], where
the number of comparisons between any two subjects follows a binomial distribution $(T, p_n)$ and $p_n$ measures the sparsity.
Under a very weak condition on the sparsity on $p_n$,
\cite{chen2019aoap} established the uniform consistency and asymptotic normality of the MLE in the Bradley--Terry model by extending the proof strategies in  \cite{simons-yao1999}.
Moreover, \cite{yan2012sparse} considered a fixed sparse comparison graph by controlling the length from one
subject to another subject with $2$ or $3$, in which the consistency and
asymptotic normality of the MLE also hold.
Inference in the high-dimensional setting under the Bradley-Terry model and some generalized versions has also attracted great interests in machine learning literature;
the upper bounds of various errors are established under different conditions [e.g.,  the $\ell_1$ error $\|\hat{\beta}-\beta \|_1$ in \cite{Spectral-Ranking-2018} and \cite{hendrickx2019graph},
the mean square error in \cite{vojnovic2017parameter}, the bias $\E\|\hat{\beta}-\beta \|_\infty$ in \cite{wang2019stretching}].
Under the assumption that the log-likelihood function is strictly convex, \cite{han2023general}
establish the uniform consistency of the MLE in general paired comparison models.
However, moment estimation with sparse comparisons has not yet explored.

The main contributions of this paper are as follows.
First, we develop the moment estimation, instead of the maximum likelihood estimation,  base on the scores of subjects (i.e., the number of wins) to estimate the merit parameters in model \eqref{model}.
The reason why we use the moment estimation is that it is natural to rank subjects according to their scores and the computation based on moment equations is simpler.
When $F(\cdot)$ belongs to the exponential family distribution, both estimations are identical.
Second, under an Erd\"{o}s--R\'{e}nyi comparison graph,
we establish a unified theoretical framework, in which the uniform consistency and asymptotic normality of the moment estimator hold
when $n$ goes to infinity and $p_n$ tends to zero.
A key idea for the proof of the consistency is that we obtain the convergence rate of the Newton iterative sequence for solving the estimator.
The asymptotic normality is proved by applying Taylor's expansions to a series of functions constructed from estimating equations and showing that
remainder terms in the expansions are  asymptotically neglect.
Although each pair of subjects is assumed to  have a
comparison with the same probability $p_n$, our proof strategy can be easily extended
to the case with different comparison probabilities at the order of $p_n$.
Third, we use the Thurstone model to illustrate the unified theoretical results.
Further extensions to a fixed sparse comparison graph in \cite{yan2012sparse} are also derived.
Numerical studies and a real data analysis illustrate our theoretical findings.

The rest of the paper is organized as follows. In Section \ref{section:model}, we present the moment estimation.
In Section \ref{section:asymptotic}, we present the consistency and asymptotic normality of the moment estimator.
We illustrate our unified results by one application in Section \ref{section:application}.
We extend the asymptotic results to a fixed comparison graph in Section \ref{section-fixed}.
In Section \ref{section:simulation}, we carry out simulations and give a real data analysis.
We give the summary and further discussion in Section \ref{section:sd}.
The proofs of the main results are relegated to Section \ref{section:appendix}.
The proofs of supported lemmas are relegated to the supplementary material.

\section{Moment estimation}
\label{section:model}

Assume that $n+1$ subjects that are labeled as ``$0, \ldots, n$", are compared in pairs repeatedly.
Let $t_{ij}$ be the times that subject $i$ compares with subject $j$ and $a_{ij}$ be the times
that subject $i$ wins subject $j$ out of $t_{ij}$ comparisons. As a result, $a_{ij} + a_{ji} = t_{ij}$.
By convention, define $t_{ii}=0$ and $a_{ii}=0$.
The comparison matrix $(t_{ij})_{(n+1)\times (n+1)}$ is generated from an Erd\"{o}s--R\'{e}nyi comparison graph,
where $t_{ij}$ follows a binomial distribution $Bin(T, p_n)$ with $p_n$ measuring the sparsity of comparisons.
More generally, $t_{ij}\sim Bin(T_{ij}, p_n)$. We set $T_{ij}$ to be the same for easy of exposition.
Recall that $\beta_0, \ldots, \beta_n$ are the merit parameters of subjects $0, \ldots, n$.
The probability in the model \eqref{model} implies that the winning probability only depends on the difference of merits of two subjects.
For the identification of models, we normalize $\beta_i, i=0,1,\ldots, n$, by setting $\beta_0=0$ as in \cite{simons-yao1999}.
We assume that all paired comparisons are independent and $a_{ij}$ follows a binomial distribution $Bin(t_{ij}, p_{ij})$ conditional on $t_{ij}$.

Let $a_i=\sum_{j=0}^n a_{ij}$ be the total wins of subject $i$ and $a=(a_1, \ldots, a_n)^\top$.
To motivate the estimating equations, we compare the maximum likelihood equation and the moment equation under the Thurstone model described in Section \ref{section:application}.
The maximum likelihood equations are
\[
\sum_{j\neq i} \left[ \frac{ a_{ij} \phi(\beta_i - \beta_j ) }{ \Phi(\beta_i - \beta_j ) }
- \frac{(t_{ij}-a_{ij})\phi(\beta_i - \beta_j ) }{ 1-\Phi(\beta_i - \beta_j ) } \right] =0,~~i=1,\ldots, n.
\]
where $\phi(\cdot)$ is the density function of the standard normality and $\Phi(\cdot)$ is its distribution function.
The corresponding  moment equations are
\begin{equation*}
a_i  =  \sum_{j\neq i} t_{ij}\Phi(\beta_i - \beta_j), ~~i=1, \ldots, n.
\end{equation*}
We can see that the latter is simpler and easier to compute. On the other hand, it is natural to rank subjects according to their scores.
Thus, we use the moment estimation here. When $F(\cdot)$ in model \eqref{model} belongs to the exponential family distributions, both are the same.

Write $\mu(\cdot)$ as the expectation of $F(\cdot)$ and $\mu_{ij}(\beta)=\mu(\beta_i - \beta_j)$.
Then the estimating equations are
\begin{equation}\label{eq:moment}
a_i  =  \sum_{j\neq i} t_{ij}\mu_{ij}(\beta),~~ i=1, \ldots, n.
\end{equation}
The solution to the above equations is the moment estimator denoted by $\widehat{\beta}=(\widehat{\beta}_1, \ldots, \widehat{\beta}_n)^\top$ and $\hat{\beta}_0=0$.
Let
\[
\varphi(\beta)= (\sum_{j\neq 1} t_{ij}\mu_{1j}(\beta), \ldots, \sum_{j\neq n} t_{nj}\mu_{nj}(\beta) )^\top.
\]
If $\varphi(\beta): \R^n \to (0, \infty)$ is an one to one mapping, then $\widehat{\beta}$ exists and is unique, i.e.,
$\widehat{\beta}=\varphi^{-1}(a)$.
When $\varphi^{-1}$ does not exist (i.e.,  $\varphi$ is not one-to-one), any solution $\widehat{\beta}$ of equation \eqref{eq:moment}
is a moment estimator of $\beta$. The Newton-Ralph algorithm can be used to solve equation \eqref{eq:moment}.
Moreover, the R language provides the package	``BradleyTerry2" to solve the estimator in the Bradley-Terry model.

We discuss the existence of $\widehat{\beta}$ from the viewpoint of graph connection.
If the comparison graph with the matrix $(t_{ij})_{i,j=0,\ldots,n}$ as its adjacency matrix is not connected,
then there are two empty sets such that there are no comparisons between subjects in the first set and those in the second.
In this case, there are no basis for ranking subjects in the first set and those in the second set.
Further, a necessary condition for the existence of $\widehat{\beta}$
is that the directed graph $\mathcal{G}_n$ with the win-loss matrix $A=(a_{ij})$ as its adjacency matrix is strongly connected.
In other words,  for every partition of the subjects into two nonempty sets, a
subject in the second set beats a subject in the first at least once.
To see this, assume that there are two empty sets $B_1$ and $B_2$ such that all subjects in $B_1$ win all comparisons with subjects in $B_2$.
Without loss of generality, we set $B_1=\{0, \ldots, m\}$ and $B_2=\{m+1, \ldots, n\}$ with $ 0\le m <n$,
where $a_{ij} = t_{ij}$ for $i\in B_1$ and $j\in B_2$. By summing $a_i$ over $i=0, \ldots, m$, we have
\[
\sum_{i=0}^m a_i = \sum_{i=0}^m \sum_{j=0}^m t_{ij} \mu(\beta_i - \beta_j ) + \sum_{i=0}^m \sum_{j=m+1}^n t_{ij}\mu(\beta_i - \beta_j).
\]
Because $a_i$ is  a sum of $a_{ij}$, $j=0,\ldots, n$, and $\mu(\beta_i - \beta_j)+\mu(\beta_j -\beta_i)=1$, we have
\[
\sum_{i=0}^m \sum_{j=m+1}^n a_{ij} =  \sum_{i=0}^m \sum_{j=m+1}^n t_{ij}\mu(\beta_i - \beta_j).
\]
Because  $a_{ij} = t_{ij}$ for $i=0, \ldots, m$ and $j=m+1, \ldots, n$ and at least such one $t_{ij}>0$, it must be $\mu(\beta_i-\beta_j)=1$ when $t_{ij}>0$ in order to guarantee both sides in the
above equation to be equal.
In this case, at least one such difference $\beta_i -\beta_j$ must go to infinity such that the moment estimate does not exist.
The strong connection of $\mathcal{G}_n$ is also sufficient for guaranteeing the existence of the MLE in the Bradley-Terry model [\cite{Ford1957}], in which the moment estimator is equal to the MLE.
It is interesting to see whether the strong connection of $\mathcal{G}_n$ is sufficient to guarantee the existence of $\widehat{\beta}$ in a general model.
In the next section, we will show that $\widehat{\beta}$ exists with probability approaching one under some mild conditions.

\section{Asymptotic properties}
\label{section:asymptotic}
In this section, we present the consistency and asymptotic normality of the moment estimator.
We first introduce some notations. For a subset $C\subset \R^n$, let $C^0$ and $\overline{C}$ denote the interior and closure of $C$, respectively.
For a vector $x=(x_1, \ldots, x_n)^\top\in \R^n$, denote $\|x\|$ by a vector norm with the $\ell_\infty$-norm, $\|x\|_\infty = \max_{1\le i\le n} |x_i|$,
and the $\ell_1$-norm, $\|x\|_1=\sum_i |x_i|$.
Let $B(x, \epsilon)=\{y: \| x-y\|_\infty \le \epsilon\}$ be an $\epsilon$-neighborhood of $x$.
For an $n\times n$ matrix $J=(J_{ij})$, let $\|J\|_\infty$ denote the matrix norm induced by the $\ell_\infty$-norm on vectors in $\R^n$, i.e.,
\[
\|J\|_\infty = \max_{x\neq 0} \frac{ \|Jx\|_\infty }{\|x\|_\infty}
=\max_{1\le i\le n}\sum_{j=1}^n |J_{ij}|,
\]
and $\|J\|$ be a general matrix norm.
Define the matrix maximum norm: $\|J\|_{\max}=\max_{i,j}|J_{ij}|$.
We use the superscript ``*" to denote the true parameter under which the data are generated.
When there is no ambiguity, we omit the superscript ``*".

Recall that $\mu(\cdot)$ is the expectation of $F(\cdot)$.
We assume that $\mu(\cdot)$ is a continuous function with the third derivative.
Write $\mu^\prime$ and $\mu^{\prime\prime}$  as the first and second derivatives of $\mu(\pi)$ on $\pi$, respectively.
Let $\epsilon_{n}$ be a small positive number.
When $\beta \in B(\beta^*, \epsilon_{n})$, we assume that there are three positive numbers, $b_{n0}, b_{n1}, b_{n2}$, such that
\begin{subequations}
\begin{gather}
\label{ineq-mu-sign}
[\min_{i,j}  \mu^\prime(\pi_{ij})]\cdot [\max_{i,j} \mu^\prime(\pi_{ij})] >0, \\
\label{ineq-mu-keya}
b_{n0}\le \min_{i,j}  |\mu^\prime(\pi_{ij})| \le \max_{i,j} |\mu^\prime(\pi_{ij})| \le b_{n1}, \\
\label{ineq-mu-keyb}
\max_{i,j}| \mu^{\prime\prime}(\pi_{ij})| \le b_{n2}, 
\end{gather}
\end{subequations}
where $\pi_{ij}:= \beta_i - \beta_j$.

We use the Bradley-Terry model to illustrate the above inequalities, where $\mu(x) = e^x/(1+e^x)$.
 A direct calculation gives that
\[
\mu^\prime(x) = \frac{e^x}{ (1+e^x)^2 },~~  \mu^{\prime\prime}(x) = \frac{e^x(1-e^x)}{ (1+e^x)^3 }.
\]

It is easy to show that
\begin{equation}\label{eq-mu-d-upper}
b_{n0}=\frac{ e^{2\|\beta^*\|_\infty + 2\epsilon_n } }{  ( 1+ e^{2\|\beta^*\|_\infty + 2\epsilon_n })^2 }\le |\mu^\prime(x)| \le b_{n1}=\frac{1}{4}, ~~ |\mu^{\prime\prime}(x)| \le b_{n2}=\frac{1}{4}.
\end{equation}
If $\epsilon_n=o(1)$, then $1/b_{n0}= O( e^{2\|\beta^*\|_\infty})$.

\subsection{Consistency}

To establish the consistency of $\widehat{\beta}$, let us first define a system of functions:
\begin{equation}\label{eqn:def:F}
H_i(\beta)= \sum\limits_{j=0}^n t_{ij}\mu_{ij}(\beta ) - a_i,~~i=0, \ldots, n,
\end{equation}
and $H(\beta)=(H_1(\beta), \ldots, H_{n}(\beta))^\top$.
It is clearly that $H(\widehat{\beta})=0$.
Let $H^\prime( \beta )$ be the Jacobian matrix of $H(\beta)$ on the parameter $\beta$.
The asymptotic behavior of $\widehat{\beta}$ depends crucially on the inverse of $H'( \beta )$.
For convenience, denote $H'( \beta )$ as $V=(v_{ij})_{i,j=1,\ldots, n}$, where
\[
v_{ij} = - t_{ij}\mu^\prime(\pi_{ij}),~ i\neq j, ~~ v_{ii}=\sum_{j=0,j\neq i}^n t_{ij} \mu^\prime(\pi_{ij}).
\]
Define
\[
v_{i0}=v_{0i}:=  \sum_{j=0, j\neq i}^{n} v_{ij} - v_{ii} = -t_{0i}\mu^\prime(  \pi_{0i} ), ~~i=1, \ldots, n, ~~
v_{00}=\sum_{j=1}^n t_{0j}\mu^\prime(  \pi_{0j} ).
\]
When $\beta\in B(\beta^*, \epsilon_n)$ and $\min_{i,j}\mu^\prime(\pi_{ij})>0$, in view of inequality \eqref{ineq-mu-keya}, the entries of $V$ satisfy the following inequalities:
\begin{equation}\label{eq1}
\begin{array}{rl}
\mbox{if~$t_{i0}>0$,~~} & t_{i0}b_{n0} \le v_{ii} + \sum_{j=1, j\neq i}^{n} v_{ij} \le t_{i0}b_{n1}, ~~i=1,\ldots, n, \\
\mbox{if~$t_{ij}>0$,~~} & t_{ij}b_{n0} \le - v_{ij} \le t_{ij}b_{n1}, ~~ i,j=1,\ldots,n; i\neq j.
\end{array}
\end{equation}
Without loss of generality, we assume that $\min_{i\neq j}\mu^\prime(\pi_{ij})>0$ when $\beta \in B(\beta^*, \epsilon_{n})$ hereafter.
(Otherwise, we redefine $H_i(\beta) = a_i - \sum_{j\neq i} \mu_{ij}(\beta)$ and repeat similar process.)
Our strategy for the proof of consistency crucially depends on
 the existence of the inverse of $V$, which requires that $V$ is a full rank matrix.
It is easy to show that $V$ is positively semi-definite.
Thus, if $V$ has a full rank, then $V$ must be positively definite.
The following lemma assures the existence of the inverse of $V$.

\begin{lemma}\label{lemma-positive}
Assume that $\min_{i,j} \mu^\prime_{ij}(\beta) > 0$. With probability at least $1 - (1-p_n)^{nT}$,
$H'(\beta)$ is positively definite.
\end{lemma}

Because $\log ( 1-x) \le -x $ when $x\in (0,1)$, we have
\[
e^{nT\log (1-p_n)} \le e^{-p_nTn}.
\]
The probability of the nonexistence of $V^{-1}$ is less than $e^{-p_nTn}$, going exponentially fast to zero.
Generally, the inverse of $V$ does not have a closed form.
\cite{Simons1998Approximating} proposed to approximate the inverse of $V$, $V^{-1}$, by the matrix
$S=(s_{ij})_{n\times n}$, where
\begin{equation}\label{definition-s}
s_{ij} = \frac{\delta_{ij}}{v_{ii}} + \frac{1}{v_{00}}.
\end{equation}
In the above equation, $\delta_{ij}=1$ if $i=j$; otherwise, $\delta_{ij}=0$.
By extending the proof of \cite{Simons1998Approximating}, to the sparse case,
the upper bound of the approximate error $\|V^{-1} - S \|_{\max}$ is given in Lemma \ref{lemma:inverse:appro}.

Review that the main idea of the proof of the consistency in the Bradley-Terry model
[\cite{simons-yao1999,yan2012sparse}] contains two parts.
Let $\hat{u}_i = e^{\hat{\beta}_i}$, $u_i = e^{\beta_i}$,
$i_0 = \arg \max_i \hat{u}_i/ u_i$ and  $i_1 = \arg \min_i \hat{u}_i/u_i$.
Since $\hat{u}_0/u_0 = 1$, it suffices to show that the ratio of subject $i_0$, $\hat{u}_{i_0}/u_{i_0}$, and
the ratio of $i_1$, $\hat{u}_{i_1}/u_{i_1}$ are very close.
With the nice mathematical properties of the logistic function $\mu(x) = e^x/(1+e^x)$,
the first part is to show that there are a number of subjects satisfying the following inequalities:
\begin{equation}\label{eq-sy-con}
 b\sum_{\{j: t_{i_0,j}>0\}} ( \hat{u}_{i_0}/u_{i_0} - \hat{u}_j/u_j ) \le c, ~~
 b\sum_{\{j: t_{i_1,j}>0\}} ( \hat{u}_j/u_j  - \hat{u}_{i_0}/u_{i_0}  ) \le c,
\end{equation}
where $b$ and $c$ are certain numbers.
The second part is to eliminate common terms $\hat{u}_j/u_j$ by using the conditions that the
number of the common neighbors between any two subjects, $\min_{i,j}\#\{k: t_{ik}>0, t_{jk}>0\}$, is at least $\tau n$, where
$\tau=1$ in \cite{simons-yao1999} and $\tau\in(0,1)$ in \cite{yan2012sparse}.
In the Erd\"{o}s-R\'{e}nyi comparison graph, \cite{chen2019aoap} further showed that there are at least one subject with its ratio close to
both $\hat{u}_{i_0}/u_{i_0}$ and $\hat{u}_{i_1}/u_{i_1}$. 

The aforementioned strategies for the proof of consistency are built on the
the premise of the existence of the MLE, which is guaranteed by the necessary and sufficient condition
that the directed graph with the win-loss matrix as its
adjacency matrix is strongly connected [\cite{Ford1957}].
As discussed before, it may be difficult to find the minimal sufficient condition to guarantee the existence of $\widehat{\beta}$ in general paired comparison models.
To overcome this difficulty, we aim to obtain the convergence rate of the Newton iterative sequence for solving equation \eqref{eq:moment}.
Under the well-known  Newton-Kantorich conditions, the Newton iterative sequence converges and its limiting point is the solution.
We apply an adjusted version of the Newton-Kantorich theorem in \cite{Yamamoto1986} to this end, which
not only guarantees the existence of the solution, but also gives an optimal error bound for the Newton iterative sequence.

Now we formally state the consistency result.

\begin{theorem}\label{Theorem:con}
Assume that conditions \eqref{ineq-mu-sign}, \eqref{ineq-mu-keya} and \eqref{ineq-mu-keyb} hold.
If $b_{n1}^4 b_{n2}/ (b_{n0}^6 p_n^4 ) = o( (n/\log n)^{1/2})$,
then $\widehat{\beta}$ exists with probability approaching one and is uniformly consistent in the sense that
\begin{equation}
\label{eq-theorem1-beta}
\| \widehat{\beta} - \beta^* \|_\infty = O_p\left( \frac{b_{n1}^2}{b_{n0}^3p_n^2} \sqrt{\frac{\log n}{n}} \right)=o_p(1).
\end{equation}
\end{theorem}

To see how small $p_n$ could be, we consider a special case that $\beta^*$ is a constant vector, in which
$b_{n0}, b_{n1}$ and $b_{n2}$ are also constants.
According to the above theorem, if $p_n> O((\log n/ n)^{1/8})$, then
$\| \widehat{\beta} - \beta^* \|_\infty=O_p( p_n^{-2} (\log n/n)^{1/2})$.

\subsection{Asymptotic normality of $\widehat{\beta}$}

We establish the asymptotic distribution of $\widehat{\beta}$ by characterizing its asymptotical representation.
In details, we apply a second order Taylor expansion to $H(\widehat{\beta})$ and find that $\widehat{\beta}-\beta^*$ can be
represented as the sum of a main term $V^{-1}(a - \E^* a )$ and an asymptotically neglect remainder term, where
$\E^*$ denotes the conditional expectation conditional on $\{t_{ij}: i,j=0,\ldots, n\}$.
Because $V^{-1}$ does not have a closed form, we use the matrix  $S$ defined in \eqref{definition-s} to approximate it.
We formally state the asymptotic normality of $\widehat{\beta}$ as follows.

\begin{theorem}\label{Theorem-central-a}
Let $V=\partial H(\beta^*)/\partial \beta$ and  $U = (u_{ij}):= \mathrm{Var}( a |t_{ij}, 0\le i,j\le n)$.
If  $b_{n2}b_{n1}^6 b_{n0}^{-9}p_n^{-6} = o( n^{1/2}/\log n)$,
then for fixed $k$, the vector $( (\widehat{\beta}_1 - \beta^*), \ldots,  (\widehat{\beta}_k - \beta^*_k))$
follows a $k$-dimensional multivariate normal distribution with mean zero and the covariance matrix $\Sigma=(\sigma_{ij})_{k\times k}$, where
\begin{equation}\label{definition-sigma}
\sigma_{ij} = \frac{\delta_{ij} u_{ii}}{v_{ii}^2 } + \frac{u_{00}}{v_{00}^2 }.
\end{equation}
\end{theorem}

\begin{remark}
If $U=V$, then $\sigma_{ij}$ is equal to $\delta_{ij}/v_{ii}+1/v_{00}$.
When $F(\cdot)$ belongs to the exponential family distribution (e.g., the Bradley-Terry model), $U$ is identical to $V$.
If $U\neq V$, then the asymptotic variance of $\widehat{\beta}_i$ is involved with an additional factor $u_{ii}$.
The asymptotic variance of $\widehat{\beta}_i$  is in the order of $(np_n)^{-1/2}$ if $\beta^*$ is bounded above by a constant.
\end{remark}

\section{Application to the Thurston model}
\label{section:application}

In this section, we illustrate the unified theoretical result by the application to the Thurston model.

The original Thurstone model has a variance $\sigma^2$ in the normal distribution,
i.e., the probability of subject $i$ is preferred over $j$ is $F((\beta_i -\beta_j)/\sigma )$. Since the merit parameters are scale invariable, we simply set $\sigma=1$ hereafter.
Recall that $\phi(x)=(2\pi)^{1/2}e^{-x^2/2}$ is the standard normal density function
and $\Phi(x)=\int_{-\infty}^{x} \phi(x) dx $ is the distribution function of the standard normality.
In the Thurston model, $\mu(x)=\Phi(x)$. Then,
\[
\mu^\prime(x) = \phi(x), \mu^{\prime\prime} (x) = \frac{ x}{\sqrt{2\pi}} e^{-x^2/2}.
\]
Since $\phi(x)=(2\pi)^{1/2}e^{-x^2/2}$ is an decreasing function on $|x|$, we have when $|x| \le Q_n$,
\[
 \frac{1}{2\pi} e^{ -Q_n^2/2} \le \phi(x) \le \frac{1}{2\pi}.
\]
Let $h(x)=xe^{-x^2/2}$. Then $h^\prime (x)= (1-x^2)e^{-x^2/2}$.
Therefore, when $x\in (0,1)$, $h(x)$ is an increasing function on its argument $x$; when $x\in (1, \infty)$,
$h(x)$ is an decreasing function on $x$. As a result, $h(x)$ attains its maximum value at $x=1$ when $x>0$.
Since $h(x)$ is a symmetric function,  we have $|h(x)|\le e^{-1/2}\approx 0.6$.
So
\[
b_{n0}=\frac{1}{2\pi} e^{ -(\|\beta^*\|_\infty + \epsilon_n)^2/2}, b_{n1}=\frac{1}{2\pi}, b_{n2}=  (2\pi e)^{-1/2}.
\]
In view of Theorems \ref{Theorem:con} and \ref{Theorem-central-a}, we have the following corollary.

\begin{corollary}
If $b_{n1}^4 b_{n2}/ (b_{n0}^6 \eta_n^2 ) = o( (n/\log n)^{1/2})$ and $p_n > 24\log n/n$,
then $\widehat{\beta}$ exists with probability approaching one and is uniformly consistent in the sense that
\begin{equation}
\label{eq-theorem1-beta}
\| \widehat{\beta} - \beta^* \|_\infty = O_p\left( \frac{b_{n1}^2}{b_{n0}^3\eta_n} \sqrt{\frac{\log n}{n}} \right)=o_p(1).
\end{equation}
(2) Let $V=\partial H(\beta^*)/\partial \beta$.
If $ b_{n2}b_{n1}^6 b_{n0}^{-9} = o( n^{1/2}/ \log n$,
then for fixed $k$, the vector $( (\widehat{\beta}_1 - \beta^*), \ldots,  (\widehat{\beta}_k - \beta^*_k))$
follows a $k$-dimensional multivariate normal distribution with  mean zero and the covariance matrix $\Sigma=(\sigma_{ij})_{k\times k}$
defined at \eqref{definition-sigma}.
\end{corollary}

\section{Extension to a fixed sparse design}
\label{section-fixed}

In some applications such as sports, the comparison graph may be fixed, not be random.
For example, in the regular season of the National Football League (NFL),  which teams having games are scheduled in advance.
More specially, there are $32$ teams in the two conferences of the NFL and are divided into eight divisions each consisting of four teams.
In the regular season, each team plays $16$ matches, $6$ within the division and $10$
between the divisions.
Motivated by the design,  \cite{yan2012sparse} proposed a sparse condition to control the length from one
subject to another subject with $2$ or $3$:
\[
\tau_n := \min_{0\le i<j \le n} \frac{\#\{k: t_{ik}>0, t_{jk}>0 \} }{t}.
\]
That is, $\tau_n$ is the minimum ratio of the total number of paths between any $i$ and $j$ with length $2$ or $3$.
Under the Erd\"{o}s--R\'{e}nyi comparison graph, there are similar sparsity. Specifically,
the set of the common neighbors of any two subjects $i$ and $j$ has at least the size
\[
\#\{k: t_{ik}>0, t_{jk}>0 \} \ge \tfrac{1}{2}(n-1)p_n^2,
\]
with probability at least $ 1- O(1/n)$ if $np_n^2 \ge 24\log n$; see \eqref{eq-xii-xiij-upper} in the proof of Lemma 3.

We assume that if two subjects have comparisons, they are compared $T$ times, in accordance with the aforementioned setting
for easy of exposition. Similar to Lemma \ref{lemma:inverse:appro},
 the approximate error of using $S$ to approximate $V^{-1}$ is
\[
\| V^{-1} - S \|_{\max} \le \frac{ 2T^2 b_{n1}^2  \rho_{\max} }{ b_{n0}^3 \tau_n^3 n^2 },
\]
where $\rho_{\max}=t_{\max}/n$ and $t_{\max}=\max_i t_i$.
With the similar lines of arguments as in the proofs of Theorems \ref{Theorem:con} and \ref{Theorem-central-a}, we have
the following theorem, whose proof is omitted.
\begin{theorem}\label{Theorem:fixed}
Assume that conditions \eqref{ineq-mu-sign}, \eqref{ineq-mu-keya} and \eqref{ineq-mu-keyb} hold.
(1)If $b_{n1}^4 b_{n2}/ (b_{n0}^6 \tau_n^2 ) = o( (n/\log n)^{1/2})$,
then $\widehat{\beta}$ exists with probability approaching one and is uniformly consistent in the sense that
\begin{equation}
\label{eq-theorem1-beta}
\| \widehat{\beta} - \beta^* \|_\infty = O_p\left( \frac{b_{n1}^2}{b_{n0}^3\tau_n} \sqrt{\frac{\log n}{n}} \right)=o_p(1).
\end{equation}
(2) Let $V=\partial H(\beta^*)/\partial \beta$ and  $U = (u_{ij}):= \mathrm{Var}( a |t_{ij}, 0\le i,j\le n)$.
If  $b_{n2}b_{n1}^6 b_{n0}^{-9}\tau_n^{-3} = o( n^{1/2}/\log n)$,
then for fixed $k$, the vector $( (\widehat{\beta}_1 - \beta^*), \ldots,  (\widehat{\beta}_k - \beta^*_k))$
follows a $k$-dimensional multivariate normal distribution with mean zero and the covariance matrix $\Sigma=(\sigma_{ij})_{k\times k}$, where
$\sigma_{ij}$ is given in \eqref{definition-s}.
\end{theorem}

\section{Numerical Studies}
\label{section:simulation}

In this section, we evaluate the asymptotic results of the moment estimator in the
Thurstone model through simulation studies and a real data example.

\subsection{Simulation studies}

We carry out simulations to evaluate the finite sample
performance of the moment estimator in the Thurstone model.
We set $T = 1$, which means that any pair has one comparison with probability
$p_n$ and no comparison with probability $1-p_n$.
Let $c$ be a constant. We set the merit parameters to be a linear form, i.e.,
$\beta_i^* = ic\log n/n$ for $i=1, \ldots, n$, where $\beta_0^*=0$.
We considered four different values for $c$ as $c=0.3, 0.5, 0.8$.
By allowing $\| \beta^*\|_\infty $ to grow with $n$, we intended to assess the asymptotic properties under different asymptotic regimes.

In order to see how small $p_n$ could be,
we first evaluate the fail frequency that the ``win-loss" graph $G_n$ is strongly connected, which is the necessary condition to guarantee the
moment estimator exists. We set $c$ to be fixed with $c=0.4$. The results are shown in Table \ref{Table:fail} with $1,000$.
We can see that the necessary condition did not hold in each simulation when $p_n=\log n/n$ while it holds with almost $100\%$ frequency when $p_n=(\log n/n )^{1/2}$.
This shows that it is necessary to control the rate of $p_n$ tending to zero.

{\renewcommand{\arraystretch}{1}
\begin{table}[!h]\centering
\caption{The fail frequency ($\times 100\%$).}
\label{Table:fail}
\begin{tabular}{clll}
\hline
$p_n$                       & $n=100$ & $n=500$ & $n=1000$  \\
\hline
$(\log n/n )^{1/2}$         & $  0.5 $ & $ 0 $ & $ 0 $  \\
$(\log n/n )^{2/3}$         & $ 25.8  $&$ 0.2 $&$ 0 $  \\
$\log n/n $                 & $ 100  $&$ 100 $&$ 100 $ \\
\hline
\end{tabular}
\end{table}
}

By Theorem \ref{Theorem-central-a},
$\hat{\xi}_{ij} = [\widehat{\beta}_i-\widehat{\beta}_j-(\beta_i^*-\beta_j^*)]/(\hat{u}_{ii}/\hat{v}_{ii}^2 + \hat{u}_{00}/\hat{v}_{00}^2)^{1/2}$
 converges in distribution to the standard normality, where $\hat{u}_{ii}$ and $\hat{v}_{i,i}$ are the estimates of $u_{ii}$ and $v_{ii}$
by replacing $\beta^*$ with $\widehat{\beta}$.
Therefore, we assessed the asymptotic normality of $\hat{\xi}_{ij}$ via the coverage probability of the $95\%$ confidence interval and the length of the confidence interval.
The times that $\widehat{\beta}$ failed to exist were also recorded.
Two values  $n=100$ and $n=200$  were considered for the number of subjects.
Each simulation was repeated $1,0000$ times.

The simulation results are shown in Table \ref{Table:alpha}.
When $p_n=(\log n/n)^{1/4}$, all simulated coverage probabilities are very close to the target level $95\%$.
On the other hand, when $p_n=(\log n/n)^{1/2}$, they are a little lower than the normal level in the case $n=100$
and are very close to $95\%$ in the case $n=200$.
The length of the confidence interval decreases as $n$ increases, which qualitatively agrees with the theory.
It is also expected that the length of the confidence interval increases as $p_n$ decreases when $n$ is fixed.
Another phenomenon is that the length of confidence interval under three distinct $c$  has little difference  when $n$ and $p_n$ are fixed.

{\renewcommand{\arraystretch}{1}
\begin{table}[!h]\centering
\caption{The reported values are the coverage frequency ($\times 100\%$) for $\beta_i-\beta_j$ for a pair $(i,j)$ / length of the confidence interval / fail probabilities ($\times 100\%$).}
\label{Table:alpha}
\begin{tabular}{clc cc}
\hline
\multicolumn{5}{c}{ $p_n= (\log n/n )^{1/4}$ } \\
\hline
n       &  $i$ & $c=0.2$ & $c=0.5$ & $c=0.8$  \\
\hline
100         &$(1,2)     $&$94.5 / 1.04 / 0 $&$ 94.62 / 1.06 / 0 $&$ 94.62 / 1.06 / 0 $  \\
            &$(49,50)   $&$94.97 / 1.04 / 0 $&$ 94.34 / 1.04 / 0 $&$ 94.34 / 1.04 / 0 $  \\
            &$(99,100)  $&$95.12 / 1.04 / 0 $&$ 94.99 / 1.06 / 0 $&$ 94.99 / 1.06 / 0 $  \\
&&&&\\
200         &$(1,2)       $&$95.18 / 0.78 / 0 $&$ 94.71 / 0.79 / 0 $&$ 94.71 / 0.79 / 0 $ \\
            &$(99,100)    $&$95.14 / 0.78 / 0 $&$ 94.49 / 0.78 / 0 $&$ 94.49 / 0.78 / 0 $  \\
            &$(199,200)   $&$94.82 / 0.78 / 0 $&$ 94.64 / 0.79 / 0 $&$ 94.64 / 0.79 / 0 $  \\
\hline
\multicolumn{5}{c}{ $p_n= (\log n/n )^{1/2}$ } \\
\hline
n       &  $i$ & $c=0.2$ & $c=0.4$ & $c=0.6$  \\
\hline
100         &$(1,2)     $&$93.28 / 1.58 / 0.19 $&$ 93.48 / 1.60 / 0.49 $&$ 93.48 / 1.60 / 1.15 $  \\
            &$(49,50)   $&$93.85 / 1.58 / 0.19 $&$ 93.94 / 1.58 / 0.49 $&$ 93.94 / 1.58 / 1.15 $  \\
            &$(99,100)  $&$93.80 / 1.58 / 0.19 $&$ 93.96 / 1.61 / 0.49 $&$ 93.96 / 1.61 / 1.15 $   \\
&&&&\\
200         &$(1,2)       $&$94.04 / 1.26 / 0 $&$ 93.89 / 1.28 / 0 $&$ 93.89 / 1.28 / 0 $ \\
            &$(99,100)    $&$94.14 / 1.26 / 0 $&$ 94.22 / 1.26 / 0 $&$ 94.22 / 1.26 / 0 $  \\
            &$(199,200)   $&$94.38 / 1.26 / 0 $&$ 93.98 / 1.28 / 0 $&$ 93.98 / 1.28 / 0 $  \\
\hline
\end{tabular}
\end{table}
}

\subsection{A real data example}

We use the 2018  NFL regular season data as an illustrated example,  which is available from
\url{https://www.espn.com/nfl/schedule/_/year/2018}.
The NFL league consists of thirty-two teams that are divided evenly into two conferences and
each conference has four divisions that have four teams each. In the regular
season, each team plays with three intra-division teams (each twice) and ten games
with ten inter-division teams (each once). As discussed in \cite{yan2012sparse},
the design of the NFL regular season satisfies the sparsity condition of the fixed comparison graph,
where $\tau_n = 1/16$. We removed two ties before our analysis.
The fitted merits that were obtained from fitting the
Thurstone model for the remaining data are given in Table \ref{Table:beta:real}, where we used ``Arizona Cardinals" with the smallest number of wins as the baseline (with $\hat{\beta}_0 = 0$).

It is interesting to compare the ordering of six playoff seeds of the two conferences
by the NFL rule with the ordering by their fitted  merits in Table \ref{Table:beta:real}.
The NFL rule are based on the regular season won-lost percentage record
(PCT) and can be briefly summarized as follows: the
teams in each division with the best PCT are seeded one through four; another two teams from
each conference are seeded five and six based on their PCT.
The six playoff seeds in the American Football Conference from No.1 to No.6 based on the PCT are
Kansas City Chiefs, New England Patriots, Pittsburgh Steelers,
Houston Texans, Los Angeles Chargers, Indianapolis Colts; while
the selected teams  based on the fitted merits  are Kansas City Chiefs, New England Patriots, Houston Texans, Baltimore Ravens,
Los Angeles Chargers,  Pittsburgh Steelers.
The corresponding six playoff seeds in the National Football Conference based on the PCT are
Los Angeles Rams, New Orleans Saints, Chicago Bears, Washington Redskins,
Seattle Seahawks, Carolina Panthers; while the selected teams base on the fitted merits are
Los Angeles Rams, New Orleans Saints, Chicago Bears, Dallas Cow boys,
Seattle Seahawks, Philadelphia Eagles.
As we can see, the selected top three teams based on the PCT and the fitted merits in each conference are the same
and the selected teams from No.4, No.5 and No.6 are not all the same.

{\renewcommand{\arraystretch}{1}
\begin{table}[!hbt]\centering
\scriptsize
\caption{The fitted merit $\hat{\beta}_i$, the number of wins $a_i$, and the standard error $\hat{\sigma}_i$.}
\label{Table:beta:real}
\begin{tabular}{llccc | llccc }
\hline
\multicolumn{5}{c|}{American Football Conference}    & \multicolumn{5}{c}{National Football Conference} \\
division   & team     &  $\hat{\beta}_i$  &  $a_i$ & $\hat{\sigma}_i$   &  division   & team     &  $\hat{\beta}_i$  &  $a_i$ & $\hat{\sigma}_i$ \\
East       & New England Patriots  & $ 1.452 $&$ 11 $&$ 0.519 $ &          East       & Dallas Cow boys      & $ 1.284 $&$ 10 $&$ 0.512 $  \\
           & New York Jets         & $ 0.338 $&$ 4 $&$ 0.530 $  &                      & Philadelphia Eagles & $ 1.193 $&$ 9 $&$ 0.511 $ \\
           & Miami Dophins         & $ 0.718 $&$ 7 $&$ 0.509 $ &                      & New York Giants     & $ 0.423 $&$ 5 $&$ 0.514 $ \\
           & Buffalo Bills         & $ 0.673 $&$ 6 $&$ 0.514 $ &                      & Washington Redskins & $ 0.756 $&$ 7 $&$ 0.511 $ \\
North      & Baltimore Ravens      & $ 1.382 $&$ 10 $&$ 0.514 $ &           North      & Chicago Bears       & $ 1.494 $&$ 12 $&$ 0.532 $ \\
           & Cincinnati Bengals    & $ 0.769 $&$ 6 $&$ 0.514 $ &                     & Green Bay Packers   & $ 0.595 $&$ 6 $&$ 0.522 $ \\
           & Pittsburgh Steelers   & $ 1.337 $&$ 9 $&$ 0.52 $ &                      & Minnesota Vikings   & $ 1.059 $&$ 8 $&$ 0.523 $  \\
           & Cleveland Browns      & $ 0.968 $&$ 7 $&$ 0.519 $ &                      & Detroit Lions       & $ 0.579 $&$ 6 $&$ 0.516 $ \\
South      & Indianapolis Colts    & $ 1.244 $&$ 10 $&$ 0.512 $ &          South      & New Orleans Saints  & $ 1.908 $&$ 13 $&$ 0.544 $ \\
           & Houston Texans        & $ 1.430 $&$ 11 $&$ 0.516 $ &                      & Atlanta Falcons     & $ 0.767 $&$ 7 $&$ 0.512 $ \\
           & Tennessee Titans      & $ 1.205 $&$ 9 $&$ 0.51 $ &                       & Carolina Panthers   & $ 0.840 $&$ 7 $&$ 0.511 $ \\
           & Jacksonville Jaguars         & $ 0.591 $&$ 5 $&$ 0.517 $ &                      & Tampa Bay Buccaneers& $ 0.506 $&$ 5 $&$ 0.52 $ \\
West       & Kansas City Chiefs    & $ 1.762 $&$ 12 $&$ 0.537 $ &            West       & Los Angeles Rams &$ 1.963 $&$ 13 $&$ 0.559 $ \\
           & Denver Broncos        & $ 0.713 $&$ 6 $&$ 0.523 $&                        & San Francisco 49ers        & $ 0.166 $&$ 4 $&$ 0.54 $ \\
           & Oakland Raiders       & $ 0.365 $&$ 4 $&$ 0.537 $&                       & Seattle Seahawks & $ 1.305 $&$ 10 $&$ 0.525 $ \\
           & Los Angeles Chargers  & $ 1.748 $&$ 12 $&$ 0.536 $ &                      & Arizona Cardinals &$ 0 $&$ 3 $&$ 0.555 $  \\
\hline
\end{tabular}

\end{table}

\section{Summary and discussion}
\label{section:sd}

We have presented the moment  estimation based on the scores of subjects in the paired comparison model under sparse comparison graphs.
We have established the uniform consistency and asymptotic normality of the moment estimator.
The consistency is shown via obtaining the convergence rate of the Newton iterative sequence.
This leads a condition on the sparsity parameter $p_n$ requiring that $p_n \ge O((\log n/n)^{1/8})$ if $\beta^*$ is a constant vector.
We note that this condition looks much stronger than that in the Bradley-Terry model in \cite{chen2019aoap}.
Since we consider a general model, it would seem to be suitable that a more severe condition is imposed.
On the other hand,  the condition imposed on $b_{n0}$ may not be best possible.
In particular, the conditions for guaranteeing the asymptotic normality seem stronger than those needed for the consistency.
Note that the asymptotic behavior of the moment estimator depends not only on $b_{n0}$, but also on the configuration of all parameters.
It would be of interest to investigate whether these conditions could be relaxed.

In this paper, we assume that given the comparison graph, all paired comparisons are independent. 
Note that the moment equation is regardless of whether comparisons are independent or not.
When comparisons are not independent, the moment estimation still works.
The consistency result in Theorem \ref{Theorem:con} still hold
as long as there are the same order of the upper bound of $a_i - \E^* a_i$ in Lemma \ref{lemma-a-upper-bound}.
In fact, the independence assumption is not directly used through checking our proofs.
It is only used in Lemma \ref{lemma-a-upper-bound} to derive the upper bound of $a_i - \E^* a_i$ by using the Hoeffding inequality.
Analogously,  the independence assumption is used to derive the central limit theorem of $a_i - \E^* a_i$.
In the dependence case, there are also a lot of Hoeffding-type exponential tail inequalities [e.g., \cite{Delyon:2009,Roussas1996,Ioannides1999423}]
and cental limit theorems for sums of a sequence of random variables (e.g. \cite{cocke1972,cox1984}) to apply.
We hope that the methods developed here can be applied to dependent paired comparison models in the future.

\section{Appendix A}
\label{section:appendix}

In this section, we present the proofs of theorems.

\subsection{Preliminaries}

In this section, we state two preliminary results, which will be used in the proofs.
The first result is the optimal error bound in the Newton method in \cite{Yamamoto1986} under the Kantorovich
 conditions [\cite{Kantorovich1948Functional}].

\begin{lemma}[\cite{Yamamoto1986}]\label{lemma:Newton:Kantovorich}
Let $X$ and $Y$ be Banach spaces, $D$ be an open convex subset of $X$ and
$F:D \subseteq X \to Y$ be Fr\'{e}chet differentiable.
Assume that, at some $x_0 \in D$, $F^\prime(x_0)$ is invertible and that
\begin{eqnarray}
\label{eq-kantororich-a}
\| F^\prime(x_0)^{-1} ( F^\prime(x) - F^\prime(y))\| \le K\|x-y\|,~~ x, y\in D, \\
\label{eq-kantororich-b}
\| F^\prime(x_0)^{-1} F(x_0) \| \le \eta, h=K\eta \le 1/2, \\
\nonumber
\bar{S}(x_0, t^*) \subseteq D, t^*=2\eta/( 1+ \sqrt{ 1-2h}).
\end{eqnarray}
Then:
(1) The Newton iterates $x_{n+1} = x_n - F^\prime (x_n)^{-1} F(x_n)$, $n\ge0$ are well-defined,
lie in $\bar{S}(x_0, t^*)$ and converge to a solution $x^*$ of $F(x)=0$. \\
(2) The solution $x^*$ is unique in $S(x_0, t^{**})\cap D$, $t^{**}=(1 + \sqrt{1-2h})/K$ if $2h<1$
and in $\bar{S}(x_0, t^{**})$ if $2h=1$. \\
(3) $\| x^* - x_n \| \le t^*$ if $n=0$ and $\| x^* - x_n \| \le 2^{1-n} (2h)^{ 2^n -1 } \eta $ if $n\ge 1$.
\end{lemma}

The second result is the approximate error of using $S$ to approximate $V^{-1}$, whose proof is in the Supplementary Material.

\begin{lemma}\label{lemma:inverse:appro}
If $p_n \ge 24\log n/n$, then for sufficiently large $n$,  with probability at least
 $ 1-  O(n^{-1})$, we have
\[
\| V^{-1} - S \|_{\max} \le \frac{12Tb_{n1}^2 }{ b_{n0}^3 n(n-1)p_n^3  }.
\]
\end{lemma}

\subsection{Proof of Theorem  \ref{Theorem:con}}

We aim to show Theorem  \ref{Theorem:con} by obtaining the convergence rate
of the Newton iterative sequence in view of Lemma \ref{lemma:Newton:Kantovorich}, which requires to
verify the  Kantovorich conditions \eqref{eq-kantororich-a} and \eqref{eq-kantororich-b}.
Condition \eqref{eq-kantororich-a} depends on the Lipschitz continuous of $H_i^\prime(\beta)$.
Recall that $t_{\max}=\max_{i=0,\ldots, n} t_i$ and $t_{\min}=\min_{i=0,\ldots, n} t_i$.

\begin{lemma}\label{lemma:lipschitz-c}
Let $D=B(\beta^*, \epsilon_{n}) (\subset \R^{n})$ be an open convex set containing the true point $\beta^*$.
For any given set $\{ t_{ij}, 0\le i, j \le n\}$, if inequality \eqref{ineq-mu-keyb} holds, then
\[
\max_{i=0, \ldots, n} \| H_i^\prime(x) - H_i^\prime(y)  \|_1 \le 4b_{n2}t_{\max}  \|x - y \|_\infty.
\]
\end{lemma}

Moreover, Condition \eqref{eq-kantororich-a} also depends on the magnitudes of $| a_i - \E (a_i|t_{ij},j=0,\ldots n) |$, $i=0, \ldots, n$,
which are stated below.

\begin{lemma}\label{lemma-a-upper-bound}
With probability at least $1-O(1/n)$, we have
\[
 \max\limits_{i=0, \ldots, n} |a_i-\E (a_i|t_{ij},j=0,\ldots, n)|\le   \sqrt{2 \log n t_{\max}} .
\]
\end{lemma}

The following results are the lower bound of $t_{\min}$, and the upper bounds of $t_{\max}$ and of $\sum_i t_i$.

\begin{lemma}\label{lemma-chernoff-bound-ti}
(1) With probability at least $ 1 - (n+1) \exp( - \tfrac{1}{8}nTp_n )$,
\[
t_{\min}=\min_{i=0, \ldots, n} t_i \ge \tfrac{1}{2}nTp_n.
\]
(2) With probability at least $ 1- (n+1) \exp( - \tfrac{1}{10} nTp_n)$,
\[
t_{\max}=\max_{i=0, \ldots, n} t_i \le  \tfrac{3}{2}nTp_n.
\]
(3) With probability at least $ 1- \exp( - \tfrac{1}{10} n(n+1)Tp_n$,
\[
\sum_{i=0}^n t_i \le  3n(n+1)Tp_n.
\]
\end{lemma}

We are now ready to prove Theorem \ref{Theorem:con}.

\begin{proof}[Proof of Theorem \ref{Theorem:con}]
Note that $\widehat{\beta}$ is the solution to
the equation $H(\beta)$=0. We prove the consistency via obtaining the convergence rate of
the Newton iterative sequence: $\beta^{(k+1)}=
\beta^{(k)} - [H^\prime(\beta^{(k)})]^{-1} H(\beta^{(k)})$, where
we set $\beta^{(0)}:=\beta^*$. To apply Lemma \ref{lemma:Newton:Kantovorich}, we choose the convex set $D = B(\beta^*, \epsilon_{n})$.
The following calculations are based on the event $E_n$:
\[
\begin{array}{c}
\{t_{ij}, 0\le i,j\le n: \max_i | a_i - \E (a_i| t_{ij}, j=0, \ldots, n) | \le   \sqrt{2t_{\max} \log n} ,
\\
H^\prime(\beta)>0,~~ t_{\min} \ge \tfrac{T}{2}np_n, ~~t_{\max} \le \tfrac{3}{2}nTp_n \}.
 \end{array}
\]

Note that $b_{n0} \le |\mu^\prime_{ij}(\beta)| \le b_{n1}$
when $\beta\in B(\beta^*, \epsilon_{n})$.
Let $V=(v_{ij})_{n\times n}=  H^\prime(\beta^*)$. We use $S$ defined in \eqref{definition-s} to approximate $V^{-1}$
and let $W=V^{-1} -S$.
We verify the Kantovorich conditions in Lemma \ref{lemma:Newton:Kantovorich} as follows.
Since $\sum_{i=0}^n H_{i}(\beta)=0$, we have
\[
\sum_{i=1}^{n} H_{ i}(\beta)= - H_{0}(\beta).
\]
By Lemma \ref{lemma:lipschitz-c},
we have
\begin{eqnarray*}
&& \| [H^\prime(\beta^*)]^{-1} [H^\prime(x)-H^\prime(y)]\|_\infty \\
& \le & \|S[H^\prime(x)-H^\prime(y)] \|_\infty + \| W[H^\prime(x)-H^\prime(y)] \|_\infty \\
& \le & \max_{i=1,\ldots, n} \frac{1 }{v_{ii}}\|H_i^\prime(x)-H_i^\prime(y)\|_1
+ \frac{1}{v_{00}}\|H_0^\prime(x)-H_0^\prime(y)\|_1 + \|W\|_\infty \|H^\prime(x)-H^\prime(y) \|_\infty \\
& \le & [\frac{2}{b_{n0} t_{\min} } + n\cdot \frac{12b_{n1}^2T}{n(n-1)b_{n0}^3p_n^3 } ]  \times 4b_{n2}t_{\max} \times \| x - y\|_\infty \\
& = & O(\frac{b_{n1}^2b_{n2}}{b_{n0}^3p_n^2})\times \| x - y\|_\infty.
\end{eqnarray*}
Thus, we can set $K= O( b_{n1}^2b_{n2}b_{n0}^{-3}\eta_n^{-1})$ in \eqref{eq-kantororich-a}.
Again, based on the event $E_n$, we have
\begin{eqnarray*}
\eta &=&\| [H'(\beta^*)]^{-1}H(\beta^*) \|_\infty \\
& \le &
n\|V^{-1} - S\|_{\max} \|H(\beta^*)\|_\infty + \max_{i=1,\ldots, n}\frac{|H_{i}(\beta^*)|}{v_{ii}}
 + \frac{|H_{0}(\beta^*)|}{v_{00}}
\\
& \le & \left[ O(\frac{ b_{n1}^2}{nb_{n0}^3p_n^3}) + O(\frac{ 1}{b_{n0}t_{\min}}) \right] \times O((t_{\max}\log n)^{1/2})  \\
& = & O\left( \frac{ b_{n1}^2}{b_{n0}^3p_n^2}\sqrt{\frac{\log n}{n}} \right).
\end{eqnarray*}
If
\begin{equation}\label{eq-keta}
K\eta = O\left( \frac{ b_{n1}^4 b_{n2} }{ b_{n0}^6 p_n^4 } \sqrt{ \frac{ \log n}{ n }} \right)=o(1),
\end{equation}
then it verifies Condition \eqref{eq-kantororich-b}.
By Lemma \ref{lemma:Newton:Kantovorich}, $\lim_k \beta^{(k)}$ exists, denoted by $\widehat{\beta}$, and it satisfies
\[
\| \widehat{\beta} - \beta^* \|_\infty = O\left( \frac{ b_{n1}^2 }{b_{n0}^3p_n^2}\sqrt{\frac{\log n}{n}} \right).
\]
By Lemmas \ref{lemma-positive}, \ref{lemma:inverse:appro}, \ref{lemma-a-upper-bound} and \ref{lemma-chernoff-bound-ti},
the event $E_n$ holds with probability at least $1-O(n^{-1})$ if $p_n \ge 24 \log n/n$.
Note that \eqref{eq-keta} implies $p_n\ge 24 \log n/n$.
It completes the proof.
\end{proof}

\subsection{Proofs for Theorem \ref{Theorem-central-a}}

Write $\mathrm{Var}^*$ and $\E^*$ as the conditional variance and conditional expectation given $t_{ij}$ for  $0\le i,j\le n$.
Let $U = (u_{ij}):= \mathrm{Var}^*( a )$. In the Bradley-Terry model, $U=H^\prime(\beta^*)$.
Note that $a_i$ is a sum of $t_i$ independent Bernoulli random variables.
By Lemma \ref{lemma-chernoff-bound-ti}, we know $\min_i t_i = O_p( np_n)$. Let $\sigma_{\min}=\min_{i\neq j} p_{ij}(1-p_{ij})$.
If $np_n \sigma_{\min} \to \infty$, then  $\min_i u_{ii} \to \infty$.
By the central limit theorem in the bound case, as in \cite{Loeve:1977} (p.289), if
$np_n \sigma_{\min} \to \infty$, then
$u_{ii}^{-1/2} \{a_i - \E^*(a_i)\}$ converges in distribution to the standard normal distribution.
When considering the asymptotic behaviors of the vector $(a_1, \ldots, a_r)$ with a fixed $r$, one could replace the degrees $a_1, \ldots, a_r$ by the independent random variables
$\tilde{a}_i=a_{i, r+1} + \ldots + a_{in}$, $i=1,\ldots,r$.
Therefore, we have the following proposition.

\begin{proposition}\label{lemma:central:poisson}

If $np_n \sigma_{\min} \to \infty$, then as $n\to\infty$,
for any fixed $r\ge 1$,  the components of $(a_1 - \E^* (a_1), \ldots, a_r - \E^* (a_r))$ are
asymptotically independent and normally distributed with variances $u_{11}, \ldots, u_{rr}$,
respectively. Moreover, the first $r$ rows of $S(a - \E^*(a))$ are asymptotically normal
with covariance matrix $\Sigma=(\sigma_{ij})$, where
\[
\sigma_{ij} = \frac{\delta_{ij} u_{ii}}{v_{ii}^2 } + \frac{u_{00}}{v_{00}^2 }.
\]
\end{proposition}

\begin{lemma}\label{lemma-cov-wh}
Let $V=H^\prime(\beta^*)$ and $W=V^{-1}-S$ and $\mathrm{Cov}^*(\cdot)=\mathrm{Cov}(\cdot|t_{ij},0\le i,j\le n)$. Then
\[
\mathrm{Cov}^*(WH^\prime(\beta^*)) =O_p( \frac{ b_{n1}^5 }{ n^2 b_{n0}^6 p_n^5 } ).
\]
Further, if $U=H^\prime(\beta^*)$, then
\[
\mathrm{Cov}^*(WH^\prime(\beta^*)) =O_p( \frac{ b_{n1}^2}{n^2b_{n0}^3 p_n^5 } ).
\]
\end{lemma}

Now, we are ready to prove Theorem \ref{Theorem-central-a}.

\begin{proof}[Proof of Theorem \ref{Theorem-central-a}]
Let $\widehat{\pi}_{ij}=\widehat{\beta}_i-\widehat{\beta}_j$ and $\pi_{ij}^*=\beta_i^* - \beta_j^*$.
By Theorem \ref{Theorem:con}, $\widehat{\beta}\in B(\beta^*,\epsilon_n)$.
To simplify notations, write  $\mu_{ij}^\prime = \mu^\prime (\pi_{ij}^*)$.
By a second order Taylor expansion, we have
\begin{equation}
\label{equ-Taylor-exp}
t_{ij}\mu( \widehat{\pi}_{ij} ) - t_{ij}\mu(\pi_{ij}^*)
=  t_{ij}\mu_{ij}^\prime (\widehat{\beta}_i-\beta_i^*)- t_{ij}\mu_{ij}^\prime (\widehat{\beta}_j-\beta_j^*) + g_{ij},~~i\neq j,
\end{equation}
where $g_{ij}$ is the second order remainder term:
\begin{eqnarray*}
g_{ij}  =   t_{ij}\mu^{\prime\prime}( \tilde{\pi}_{ij} ) [(\widehat{\beta}_i-\beta_i)^2 +  (\widehat{\beta}_j-\beta_j)^2 - 2(\widehat{\beta}_i-\beta_i)(\widehat{\beta}_j-\beta_j)].
\end{eqnarray*}
In the above equation, $\tilde{\pi}_{ij}$ lies between $\pi_{ij}^*$ and $\widehat{\pi}_{ij}$.
If $b_{n1}^4 b_{n2} b_{n0}^{-6}p_n^{-4} = o( (n/\log n)^{1/2})$, by Theorem \ref{Theorem:con}, we have
\[
\| \widehat{\beta}-\beta^* \|_\infty = O_p\left( \frac{ b_{n1}^2 }{b_{n0}^3p_n^2}\sqrt{\frac{\log n}{n}} \right).
\]
Therefore, in view of \eqref{ineq-mu-keyb}, $|\mu^{\prime\prime}_{ij}(\tilde{\pi}_{ij})|\le t_{ij}b_{n2}$ such that
\begin{equation}
\label{inequality-gij}
|g_{ij}|  \le   4b_{n2}t_{ij} \| \widehat{\beta} - \beta^*\|_\infty^2.
\end{equation}
Let  $g_i=\sum_{j\neq i} g_{ij}$, $i=0, \ldots, n$, and $g=(g_1, \ldots, g_{n})^\top$.
Then, by Lemma \ref{lemma-chernoff-bound-ti} (2), we have
\begin{equation}\label{eq-g-upper-bound}
\max_{i=0,\ldots, n} |g_i| =
4b_{n2} t_{\max}\cdot
O_p\left( \frac{ b_{n1}^4\log n}{nb_{n0}^6 p_n^4} \right) = O_p\left( \frac{ b_{n1}^4 b_{n2} \log n}{b_{n0}^6 p_n^3} \right).
\end{equation}
By writing the equation in \eqref{equ-Taylor-exp} into a matrix form, we have
\begin{equation}\label{eq-d-formu}
 \E^* a -a  = V(\widehat{\beta} - \beta^*) + g.
\end{equation}
Equivalently,
\begin{equation}
\label{expression-beta}
\widehat{\beta} - \beta^* = V^{-1}( \E^* a - a)  + V^{-1} g.
\end{equation}
Similarily, we have
\begin{eqnarray*}
\E^* a_0 - a_0 & = & \frac{\partial H_0(\beta^*) }{ \partial \beta} (\widehat{\beta} - \beta^*)
+ \frac{1}{2} (\widehat{\beta} - \beta^*)^\top \frac{\partial^2 H_0(\tilde{\beta}) }{ \partial \beta \partial \beta^\top }
(\widehat{\beta} - \beta^*) \\
& = & -\sum_{i=1}^n v_{i0}(\widehat{\beta}_i - \beta_i^*) + \frac{1}{2} v_{i0}  (\widehat{\beta}_i - \beta_i^*)^2.
\end{eqnarray*}
Therefore, by Lemma \ref{lemma-chernoff-bound-ti} (2), we have
\[
|\E^* a_0 - a_0 + \sum_{i=1}^n v_{i0}(\widehat{\beta}_i - \beta_i^*) | = O_p\left( \frac{ b_{n1}^4t_{\max}\log n}{nb_{n0}^6 p_n^4} \right)
=O_p\left( \frac{ b_{n1}^4\log n}{ b_{n0}^6 p_n^3} \right).
\]

Note that $\sum_{i=0}^n a_i- \E^* (a_i)=0$.
By left multiplying both sided of \eqref{eq-d-formu} by a row vector with all element $1$, it yields
\[
\sum_{i=1}^n g_i = a_0- \E^* a_0 + \sum_{i=1}^n v_{i0}(\widehat{\beta}_i - \beta_i^*).
\]
So, we have
\begin{equation}\label{eq-sum-R}
|\sum_{i=1}^n g_i| = O_p\left( \frac{ b_{n1}^4 \log n}{b_{n0}^6 p_n^3} \right).
\end{equation}
By \eqref{eq-g-upper-bound} and Lemma \ref{lemma:inverse:appro}, we have
\begin{eqnarray*}
\|V^{-1}g \|_\infty & \le & \|S g \|_\infty + \|(V^{-1}- S) g \|_\infty
\\
& \le &  \max_{i=1, \ldots, n}\frac{1}{v_{ii}}|g_i|  + \frac{1}{v_{00}} |\sum_{i=1}^{n}g_i|+
 n \| V^{-1} - S\|_{\max} \|g\|_\infty  \\
& = &  O_p\left(  \frac{b_{n2}}{t_{\min}b_{n0}}  +  \frac{b_{n1}^2}{nb_{n0}^3p_n^3 } \right)  \times O_p\left( \frac{ b_{n1}^4
b_{n2} \log n}{b_{n0}^6 p_n^3} \right) \\
& = & O_p( \frac{b_{n2} b_{n1}^6 \log n}{nb_{n0}^9p_n^6}  ).
\end{eqnarray*}
If $b_{n2}b_{n1}^6 b_{n0}^{-9}p_n^{-6} = o( n^{1/2}/\log n)$, then we have
\begin{equation}
\label{eq-beta-i-expansion}
\widehat{\beta}_i - \beta^*_i = V^{-1}( \E^* a - a) + o_p(n^{-1/2}).
\end{equation}
Consequently, in view of Lemma \ref{lemma-cov-wh}, we have
\[
\widehat{\beta}_i - \beta^*_i = [S( \E^* a - a)]_i  + o_p(n^{-1/2}).
\]
Therefore, Theorem \ref{Theorem-central-a} immediately comes from Proposition \ref{lemma:central:poisson}. 
\end{proof}

\subsection{Appendix B}

In this section, we present the proofs of supported lemmas.

\section{Proof of Lemma \ref{lemma-positive}}
\label{section-lemma1}

\begin{proof}[Proof of Lemma \ref{lemma-positive}]
For an arbitrarily given nonzero vector $x=(x_1, \ldots, x_n)^\top\in \R^n$, direct calculations give
\begin{eqnarray*}
x^\top V x & = &   = \sum_{i=1}^n x_i^2 v_{ii} + \sum_{i=1}^n \sum_{j=1, j\neq i}^n x_i v_{ij} x_j \\
& = & - \sum_{i=1}^n \sum_{j=1,j\neq i}^n x_i^2 v_{ij} - \sum_{i=1}^n x_i^2 v_{i0}  + \sum_{i=1}^n \sum_{j=1, j\neq i}^n x_i v_{ij} x_j \\
& = & -\tfrac{1}{2} \sum_{i=1}^n \sum_{j=1,j\neq i}^n (x_i-x_j)^2 v_{ij} - \sum_{i=1}^n x_i^2 v_{i0},
\end{eqnarray*}
where the second equality is due to that $v_{ii}=-\sum_{j\neq i} v_{ij}$.
Therefore, $x^\top V x=0$ if and  only if
\[
x_i v_{i0}=0, ~~ i=1, \ldots, n, v_{ij}(x_i - x_j)=0,  1\le i\neq j \le n.
\]
Because $\mu_{ij}(\beta)\neq 0$ and $v_{ij}=t_{ij}\mu_{ij}(\beta)$ for $i\neq j$, the above equations are identical to
\[
x_i t_{i0}=0, ~~ i=1, \ldots, n, t_{ij}(x_i - x_j)=0,  1\le i\neq j \le n.
\]
Let $E$ be the event
\[
\{ \{t_{ij}\}_{i,j=0, i\neq j}^n:  x_i t_{i0}=0, ~~ i=1, \ldots, n, t_{ij}(x_i - x_j)=0,  1\le i\neq j \le n \}.
\]
To show Lemma 1, it is sufficient to obtain the lower bound of the probability of the event $E$.
We will evaluate  the probability of the event $E$ under two cases:
there exists some zero element in $x$ and there are no zero elements in $x$.

Case I: We consider there are zero elements in $x$. Let $\{0, x_{i_1}, \ldots, x_{i_k} \}$
be $k+1$ different distinct values in $\{x_1, \ldots, x_n\}$,
 $\Omega_0 =\{ i: x_i =0 \}$ and $\Omega_j= \{ q: x_{q} = x_{i_j} \}$, $j=1, \ldots, k$.
Since $x\neq 0$, $k\ge 1$. It is clear that
\begin{equation}\label{eq-lem-1-condtion}
|\Omega_0| >0, ~~|\Omega_j|>0, j=1, \ldots, k, ~~ \sum_{i=0}^k |\Omega_i| = n+1,
\end{equation}
where $|\Omega_i|$ denotes the cardinality of $\Omega_i$.
Therefore we have
\begin{eqnarray*}
\P( E ) & = & ( 1- p_n)^{ \sum_{j=1}^k |\Omega_j | } \times \prod_{0\le i < j \le k} ( 1-p_n)^{T |\Omega_i||\Omega_j| } \\
& = & ( 1-p_n)^{ \sum_{j=1}^k |\Omega_j| + \sum_{0 \le i < j \le  k } |\Omega_i| |\Omega_j | }.
\end{eqnarray*}
To obtain the lower bound of $\P(E)$, it is sufficient to solve
 the minimizer of $\sum_{j=1}^k |\Omega_j| + \sum_{0 \le i < j \le  k } |\Omega_i| |\Omega_j |$
under the restricted condition \eqref{eq-lem-1-condtion}. Let $y_i = |\Omega_i|$. Then,
\begin{align*}
   & \sum_{j=1}^k y_i + \sum_{0 \le i < j \le  k } y_i y_j \\
 = & \sum_{j=1}^k y_i + \tfrac{1}{2} \sum_{i=0}^n \sum_{j=0, j\neq i}^k y_i y_j \\
 = & \tfrac{1}{2} \sum_{i=0}^n \sum_{j=0}^n y_i y_j - \tfrac{1}{2} \sum_{i=0}^k y_i^2 + \sum_{j=1}^k y_i \\
 = & \tfrac{1}{2} (\sum_{i=0}^n y_i )^2 - \tfrac{1}{2} \sum_{i=1}^k (y_i-1)^2 - \tfrac{1}{2} y_0^2 +  \tfrac{1}{2}k  \\
 = & \tfrac{1}{2}( (n+1)^2 +k ) - \tfrac{1}{2} (\sum_{i=0}^k z_i^2),
\end{align*}
where $z_0=y_0$ and $z_i = y_i - 1$, $i=1, \ldots, k$.
Under the restriction $\sum_i z_i = n+1-k>0$ and $z_i \ge 0$, the function  $\sum_{i=0}^k z_i^2$ obtains its maximizer
at such points $z=(0, \ldots, n+1-k, 0, \ldots, 0)$.
So we have
\begin{align*}
& \sum_{j=1}^k y_i + \sum_{0 \le i < j \le  k } y_i y_j  \\
\ge & \tfrac{1}{2} ((n+1)^2 + k - (n+1-k)^2 )= \tfrac{1}{2} (2(n+1)k + k - k^2 ) \\
= & \tfrac{1}{2} [- ( k- \frac{2(n+1)+1}{2})^2 + (n+1 + \frac{1}{2})^2 ].
\end{align*}
Because $1\le k\le n$, the above function obtains its minimizer at $k=1$. That is,
\begin{align*}
 & - ( k- \frac{2(n+1)+1}{2})^2 + ((n+1) + \frac{1}{2})^2 \\
 \ge &  -( 1 - (n+1) - 1/2)^2 + (n+1)^2 + (n+1) + 1/4 \\
 = & 2(n+1).
\end{align*}
This shows
\[
\P(E) \le (1-p_n)^{2T(n+1)}.
\]

Case II: there are no zero elements in $x$. With the same notation $\Omega_j$ as in Case I,  we have that $|\omega_0|=0$ and
\begin{eqnarray*}
\P( E ) & = & ( 1- p_n)^{ \sum_{j=1}^k |\Omega_j | } \times \prod_{1\le i < j \le k} ( 1-p_n)^{T |\Omega_i||\Omega_j| } \\
& = & ( 1-p_n)^{ T(\sum_{j=1}^k |\Omega_j| + \sum_{1 \le i < j \le  k } |\Omega_i| |\Omega_j | )}.
\end{eqnarray*}
It is sufficient to obtain the minimizer of $\sum_{j=1}^k |\Omega_j| + \sum_{1 \le i < j \le  k } |\Omega_i| |\Omega_j |$.
under the restriction $\sum_i |\Omega_i| = n $ and $k\ge 1$. Let $y_i = |\Omega_i|$. Then,
\begin{eqnarray*}
 \sum_{j=1}^k y_i + \sum_{1 \le i < j \le  k } y_i y_j
& = & \sum_{j=1}^k y_i + \tfrac{1}{2} \sum_{i=1}^n \sum_{j=1, j\neq i}^k y_i y_j \\
& = & \tfrac{1}{2} \sum_{i=1}^k \sum_{j=1}^k y_i y_j - \tfrac{1}{2} \sum_{i=1}^k y_i^2 + (n+1) \\
& = & \tfrac{1}{2} (\sum_{i=1}^k y_i )^2 - \tfrac{1}{2} \sum_{i=1}^k y_i^2 + (n+1)  \\
& = & \tfrac{1}{2} (n+1)^2 + (n+1)  - \tfrac{1}{2} (\sum_{i=1}^k y_i^2).
\end{eqnarray*}
Under the restriction $\sum_i y_i = n+1>0$ and $z_i \ge 0$, the functions  $\sum_{i=1}^k z_i^2$ obtain its maximizer
at such points $z=(0, \ldots, n+1, 0, \ldots, 0)$.
So we have
\[
\sum_{j=1}^k y_i + \sum_{1 \le i < j \le  k } y_i y_j \ge \tfrac{1}{2} ( (n+1)^2 + (n+1) - (n+1) ^2 )\ge n+1.
\]
This shows
\[
\P(E) \le ( 1- p_n)^{(n+1)T}.
\]

By combining the lower bounds of $\P(E)$ under Cases I and II, we have
\[
\P(E^c) \ge 1 - ( 1- p_n)^{(n+1)T},
\]
where $E^c$ denotes that $V$ is positively definite.
\end{proof}

\subsection{Proof of Lemma \ref{lemma:inverse:appro}}
\label{section-lemma3}

\begin{proof}[Proof of Lemma \ref{lemma:inverse:appro}]
By Lemma 1, $V$ is a positively definite matrix with probability at least $1-e^{-p_n T(n+1)}$.
In what follows, we assume that $V$ is positively definite such that its inverse exists.
The proof proceeds two parts. The first part is to evaluate the cardinality  of the set of the common neighbors of any two subjects $i$ and $j$.
That is, we establish the lower bound:
\[
\min_{i,j} \#\{k: t_{i k}>0, t_{k j}>0\}.
\]
The second part is to show such inequality [c.f. \eqref{eq-cancel}]
\[
a \ge b [\sum_{ \{k: t_{\alpha k}>0, t_{k \beta}>0\} } (z_{i\alpha} - z_{i\beta}) ].
\]
We use the method of the proof in \cite{Simons1998Approximating} with minor modifications
that simplify their proofs, to show the second part.

Part I. Let $1_{\{\cdot\}}$ be an indicator variable.
It is equal to one when the expression in $\{ \cdot\}$ is true; otherwise, it is equal to zero.
For any given $i\neq j$, define
\[
\xi_{ij} = \sum_{k=0, k\neq i,j}^n  1_{\{ t_{ik}>0, t_{jk}>0\} }.
\]
Note that $\xi_{ij}$ is the sum of $n-1$ independent Bernoulli random variables and for three distinct indices $i,j,k$,
\[
\E \xi_{ij}= \P( t_{ik}>0) \P( t_{jk}>0 ) = ( 1 - (1-p_n)^T )^2 :=\eta_n.
\]
By the Chernoff bound in \cite{chernoff1952}, we have
\begin{eqnarray*}
\P\left(  \xi_{ij} \le \tfrac{1}{2} (n-1) \eta_n \right)  \le   \exp\left( - \tfrac{1}{8} (n-1) \eta_n    \right).
\end{eqnarray*}
It follows that
\begin{eqnarray*}
\P\left(  \min_{i,j}\xi_{ij} \le \tfrac{1}{2} (n-1) \eta_n \right)  \le  \sum_{i,j}\P\left(  \xi_{ij} \le \tfrac{1}{2} (n-1) \eta_n \right)
\le \tfrac{(n+1)n}{2} \exp\left( - \tfrac{1}{8} (n-1) \eta_n    \right).
\end{eqnarray*}
Since $T\ge 1$,
\[
\eta_n = ( 1 - (1-p_n)^T)^2 \ge ( 1 - ( 1- p_n))^2 = p_n^2.
\]
That is, with probability at least $ 1-  \tfrac{(n+1)n}{2} \exp\left( - \tfrac{1}{8} (n-1) \eta_n    \right)$, we have
\begin{equation}\label{eq-xii-xiij-upper}
\min_{i,j}\xi_{ij}  \ge  \tfrac{1}{2} (n-1) \eta_n \ge \tfrac{1}{2}(n-1)p_n^2.
\end{equation}

Part II. For convenience, we introduce a nonnegative array $\{q_{ij} \}_{i,j=1}^n$, where
\[
q_{ij}:=-v_{ij},~i\neq j; ~~ q_{ii} := \sum_{k=1}^n v_{ik} = - v_{i0}, ~~i,j=1, \ldots, n.
\]
Let
\begin{eqnarray*}
m:= \min_{(i,j) \in \{ (i,j): q_{ij}>0 \}} q_{ij}, && M:= \max_{i,j} q_{ij}, \\
t_{\max}:=\max_i t_i, && t_{\min}:=\min_i t_i.
\end{eqnarray*}
It is clear that $M\ge m>0$ and
\[
q_{ij}\ge 0, ~~ q_{ij}=q_{ji}, ~~ v_{ij}=-q_{ij}, ~i\neq j, ~~ M t_{\max} \ge  v_{ii}= \sum_{k=1}^n q_{ik} \ge m t_{\min}.
\]

Notice that
\[
V^{-1} - S = (V^{-1} - S)(I - VS ) + S(I - VS),
\]
where $I$ is a $n\times n$ identity matrix. Let $X = I - VS$, $Y=SX$ and $ Z= V^{-1} -S$, we have
the recursion formula
\begin{equation*}\label{eq-recursion}
Z = ZX + Y.
\end{equation*}
The goal is to give an upper bound of all $|z_{ij}|$.

According the definition of $S$, $V$ and $X$, we have
\begin{eqnarray*}
x_{ij} & = & \delta_{ij} - \sum_{k=1}^n v_{ik} s_{kj}  \\
       & = & \delta_{ij} - \sum_{k=1}^n v_{ik} ( \frac{ \delta{kj}}{v_{jj}} + \frac{1}{v_{00}} ) \\
       & = & \delta_{ij} - \frac{v_{ij}}{v_{ii}} - \frac{ q_{ii} }{ v_{00} } \\
       & = & ( 1-\delta_{ij}) \frac{ q_{ij}}{v_{jj} } - \frac{ q_{ii}}{v_{00}}.
\end{eqnarray*}
and
\begin{eqnarray*}
y_{ij} & = & \sum_{k=1}^n s_{ik} x_{kj} = \sum_{k=1}^n ( \frac{ \delta_{ik}}{v_{ii}} + \frac{1}{v_{00}} )
( (1-\delta_{kj})\frac{q_{kj}}{v_{jj}} - \frac{ q_{kk}}{v_{00}} ) \\
& = & \sum_{k=1}^n \frac{ \delta_{ik}}{v_{ii}} \left\{ ( 1-\delta_{kj})\frac{ q_{kj}}{v_{jj}} -\frac{ q_{kk}}{v_{00}} \right\}
+ \sum_{k=1}^n \frac{1}{v_{00}} \left\{  (1 - \delta_{kj}) \frac{ q_{kj}}{ v_{jj}} - \frac{ q_{kk}}{v_{00}} \right\} \\
& = & \frac{ ( 1-\delta_{ij}) q_{ij} }{ v_{ii}v_{jj} } - \frac{ q_{ii} }{ v_{ii} v_{00} } - \frac{ q_{jj} }{ v_{jj}v_{00} }.
\end{eqnarray*}
Since
\[
0 \le \frac{ q_{ij}}{v_{ii}v_{jj}} \le \frac{ M}{ m^2 t_{\min}^2}, ~~
0 \le \frac{ q_{ij}}{v_{ii}v_{00}} \le \frac{ M}{ m^2 t_{\min}^2 },
\]
for any different $i,j,k$, we have
\begin{equation}\label{eq-yij-upper}
|y_{ij} | \le a:=\frac{ 2M}{ m^2 t_{\min}^2 }, ~~ | y_{ij} - y_{ik} | \le a.
\end{equation}
In view of the expressions of $x_{ij}$ and $y_{ij}$, we have
\begin{equation} \label{eq-zij}
z_{ij} = \sum_{k=1}^n z_{ik} ( 1 - \delta_{kj}) \frac{ q_{kj}}{v_{jj}} - \sum_{k=1}^n z_{ik} \frac{ q_{kk}}{v_{00}} + y_{ij},~~
i,j=1, \ldots, n.
\end{equation}
Now we fixe an arbitrary $i$ and consider the upper bound of $\max_{j} |z_{ij}|$.

Let $\alpha$ and $\beta$ make that
\[
z_{i\alpha} = \max_{k=1, \ldots,n} z_{ik}, ~~ z_{i\beta } = \min_{k=1, \ldots, n} z_{ik}.
\]
Without loss of generality, we assume $z_{i\alpha} \ge | z_{i\beta} |$.
(Otherwise, we can inverse the sign of $z_{ik}$ and repeat the same process.)
Below, we will show $z_{i\beta} \le 0$. Note that this conclusion is not investigated in \cite{Simons1998Approximating}.
By multiplying $v_{jj}$ to both sides of \eqref{eq-zij}, we have
\[
v_{jj} z_{ij} = \sum_{k=1}^n z_{ik} ( 1-\delta_{kj}) q_{kj} - \sum_{k=1}^n z_{ik} \frac{ q_{kk}v_{jj}}{v_{00}} + v_{jj}y_{ij}.
\]
By summarizing the above equations with $j=1, \ldots, n$, we have
\[
\sum_{j=1}^n v_{jj} z_{ij} = \sum_{k=1}^n \sum_{j=1}^n z_{ik} ( 1-\delta_{kj}) q_{kj} - \sum_{k=1}^n
 \sum_{j=1}^n z_{ik} \frac{ q_{kk}v_{jj}}{v_{00}} +
 \sum_{j=1}^n v_{jj}( \frac{ ( 1-\delta_{ij}) q_{ij} }{ v_{ii}v_{jj} } - \frac{ q_{ii} }{ v_{ii} v_{00} } - \frac{ q_{jj} }{ v_{jj}v_{00} }).
\]
Thus,
\begin{eqnarray*}
&&\sum_{k=1}^n \sum_{j=1}^n z_{ik} \frac{ q_{kk}v_{jj}}{v_{00}} + \sum_{k=1}^n z_{ik} q_{kk} \\
& = & \sum_{j=1}^n v_{jj} ( \frac{ ( 1-\delta_{ij}) q_{ij} }{ v_{ii} v_{jj}} - \frac{ q_{ii}}{v_{ii} v_{00} } -\frac{ q_{jj} }{v_{jj}v_{00}} ) \\
& = & \sum_{j=1}^n \frac{ ( 1-\delta_{ij}) q_{ij}}{v_{ii}} - \frac{ q_{ii} \sum_j v_{jj}}{ v_{ii}v_{00} }
 - \sum_{j=1}^n \frac{ q_{jj} }{v_{00}} \\
& = & - \frac{q_{ii}}{v_{ii}} - \frac{ q_{ii} \sum_j v_{jj}}{v_{ii}v_{00}} = -\frac{q_{ii}}{v_{ii}v_{00}} ( v_{00} + \sum_{j=1}^n v_{jj})
\end{eqnarray*}
Thus,
\begin{eqnarray*}
- \frac{ q_{ii}}{v_{ii}v_{00}} ( v_{00} + \sum_{j=1}^n v_{jj}) & = &
 \sum_{k=1}^n z_{ik} \frac{ q_{kk} \sum_{j=1}^n v_{jj}}{ v_{00} }  + \sum_{k=1}^n z_{ik} q_{kk} \\
 & \ge & z_{i\beta} (v_{00} + \sum_{j=1}^n v_{jj}),
\end{eqnarray*}
This shows
\begin{equation}
\label{eq-zibeta-zero}
z_{i\beta} \le - \frac{ q_{ii}}{v_{ii}v_{00}} \le 0.
\end{equation}

Since $\sum_{k=1}^n q_{k\alpha} /v_{\alpha\alpha}=1$, we have
\[
\sum_{k=1}^n z_{i\alpha} \frac{ q_{k\alpha} }{ v_{\alpha\alpha} }
= \sum_{k=1}^n z_{ik} ( 1-\delta_{k\alpha}) \frac{ q_{k\alpha} }{v_{\alpha\alpha}} - \sum_{k=1}^n z_{ik} \frac{q_{kk}}{v_{00}} + y_{i\alpha}.
\]
In other words,
\[
\sum_{k=1}^n [ z_{i\alpha} - z_{ik}(1-\delta_{k\alpha}) ] \frac{ q_{k\alpha}}{v_{\alpha\alpha}}
= - \sum_{k=1}^n z_{ik} \frac{ q_{kk}}{v_{00}} + y_{i\alpha}.
\]
Analogously, we have
\[
\sum_{k=1}^n [ z_{ik}(1-\delta_{k\beta}) - z_{i\beta}  ] \frac{ q_{k\beta}}{v_{\beta\beta}}
= \sum_{k=1}^n z_{ik} \frac{ q_{kk}}{v_{00}} - y_{i\beta}.
\]
Therefore,
\begin{eqnarray}
\nonumber
y_{i\alpha}-y_{i\beta} & = & \sum_{k=1}^n \{ [ z_{i\alpha} - z_{ik}(1-\delta_{k\alpha}) ] \frac{ q_{k\alpha}}{v_{\alpha\alpha}}
+ [ z_{ik}(1-\delta_{k\beta}) - z_{i\beta}  ] \frac{ q_{k\beta}}{v_{\beta\beta}} \} \\
\label{eq-cancel}
& \ge & [\sum_{ \{k: t_{\alpha k}>0, t_{k \beta}>0\} } (z_{i\alpha} - z_{i\beta}) ] \times \frac{ m}{ M t_{\max} }.
\end{eqnarray}
The following calculations are base on the event $E_n$:
\renewcommand{\arraystretch}{1.2}
\begin{equation*}
\{ \min_{i\neq j} \xi_{ij} \ge \tfrac{1}{2} (n-1)p_n^2, ~~ t_{\min} \ge \tfrac{1}{2}nTp_n, ~~ t_{\max} \le \tfrac{3}{2}nTp_n \}
\end{equation*}
In view of \eqref{eq-zibeta-zero} and \eqref{eq-cancel}, we have
\begin{eqnarray}
\nonumber
z_{i\alpha} & \le  & z_{i\alpha} - z_{i\beta}  \le  \frac{ M t_{\max} }{m} \times  [\min_{i\neq j} \xi_{ij}]^{-1} \times \frac{ 2M }{m^2 t_{\min}^2  } \\
\nonumber
& \le &  \frac{ M\cdot \tfrac{3}{2}nTp_n  }{ m} \times \frac{1}{ \tfrac{1}{2} (n-1) p_n^2 } \times   \frac{ 2M }{m^2 (\tfrac{1}{2}nTp_n)^2  } \\
\label{ineq-zialpha-ab}
& = & \frac{ 24M^2}{ n(n-1)m^3 Tp_n^3 }
\end{eqnarray}
Note that $M=Tb_{n1}$ and $m= b_{n0}$.
By Lemma 6 in the main text, we have
\[
\P( t_{\min} \ge \tfrac{1}{2}nTp_n, t_{\max} \le \tfrac{3}{2}nTp_n ) \ge
1- 2(n+1)\exp( -\tfrac{1}{8}nTp_n).
\]
In view of inequality \eqref{eq-xii-xiij-upper},
we have
\[
P(E_n) \ge 1-  \tfrac{(n+1)n}{2}e^{-\tfrac{1}{8}(n-1)p_n} - 2(n+1) e^{-\tfrac{1}{8} nTp_n }.
\]
If $p_n \ge 24\log n/n$, then
\[
\tfrac{(n+1)n}{2}e^{-\tfrac{1}{8}(n-1)p_n} \le O(\frac{1}{n}),~~(n+1) e^{-\tfrac{1}{10} np_n }=o(\frac{1}{n^{1.4}}),
\]
such as
\[
P(E_n) \ge  1 - O(\frac{1}{n}).
\]
Let $F_n$ be the event that $V^{-1}$ exists. By Lemma 1, if $p_n \ge 24\log n/n$, then
\[
\P(F_n) \ge 1 - \exp( p_n T(n+1) ) \ge 1 - \frac{1}{n^{24}}.
\]
Therefore,
\[
\P(E_n\bigcap F_n) \ge 1 - O(\frac{1}{n}).
\]
Substituting $M=Tb_{n1}$ and $m=b_{n0}$ into \eqref{ineq-zialpha-ab}, it shows Lemma 2.
\end{proof}

\subsection{Proof of Lemma \ref{lemma:lipschitz-c}}
\label{section-lemma4}

\begin{proof}[Proof of Lemma \ref{lemma:lipschitz-c}]
Recall that $\pi_{ij}=\beta_i - \beta_j$ and
\[
H_i(\beta) =  \sum_{j\neq i} t_{ij} \mu(\pi_{ij}) - a_i, ~~ i=1, \ldots, n.
\]
The Jacobian matrix $H^\prime(\beta)$ of $H(\beta)$ can be calculated as follows.
By finding the partial derivative of $H_i$ with respect to $\beta$ for $i\neq j$, we have
\[
\frac{\partial H_i(\beta) }{ \partial \beta_j} = - t_{ij} \mu^\prime (\pi_{ij}), ~~
\frac{ \partial H_i(\beta)}{ \partial \beta_i} =  \sum_{j\neq i} t_{ij} \mu^\prime (\pi_{ij}),
\]
\[
\frac{\partial^2 H_i(\beta) }{ \partial \beta_i \partial \beta_j} = - t_{ij}\mu^{\prime\prime} (\pi_{ij}),~~
\frac{ \partial^2 H_i(\beta)}{\partial \beta_i^2} =  \sum_{j\neq i} t_{ij} \mu^{\prime\prime} (\pi_{ij}).
\]
When $\beta\in B(\beta^*, \epsilon_{n})$, by condition \eqref{ineq-mu-keyb},  we have
\[
|\frac{\partial^2 H_i(\beta) }{ \partial \beta_i \partial \beta_j}|\le b_{n2}t_{ij},~~i\neq j.
\]
Let
\[
\mathbf{g}_{ij}(\beta)=(\frac{\partial^2 H_i(\beta) }{ \partial \beta_1 \partial \beta_j}, \ldots,
\frac{\partial^2 H_i(\beta) }{ \partial \beta_n \partial \beta_j})^\top.
\]
Therefore,
\begin{equation}
\label{inequ:second:deri}
|\frac{\partial^2 H_i(\beta) }{\partial \beta_i^2} |\le t_ib_{n2},~~
|\frac{\partial^2 H_i(\beta) }{\partial \beta_j\partial \beta_i}| \le b_{n2}t_{ij}.
\end{equation}
It leads to that
$\|\mathbf{g}_{ii}(\beta)\|_1 \le 2t_ib_{n2}$.
Note that when $i\neq j$ and $k\neq i, j$,
\[
\frac{\partial^2 H_i(\beta) }{ \partial \beta_k \partial \beta_j} =0.
\]
Therefore, we have
$\|\mathbf{g}_{ij}(\beta)\|_1 \le 2t_{ij}b_{n2}$, for $j\neq i$. Consequently, for vectors $x, y\subset D$, we have
\begin{eqnarray*}
& & \max_{i=0, \ldots, n} \| H_i^\prime(x) - H_i^\prime(y)  \|_1 \\
& \le & \max_{i=0, \ldots, n}\sum_{j=1}^n | \frac{\partial H_i(x)}{\partial x_j } - \frac{\partial H_i(y)}{\partial y_j } | \\
& = & \max_{i=0, \ldots, n} \sum_{j=1}^n |\int_0^1 [\mathbf{g}_{ij}(tx+(1-t)y)]^\top (x-y)dt | \\
& \le & \max_{i=0, \ldots, n} 4b_{n2}t_i \|x-y\|_\infty \\
& = & 4b_{n2}t_{\max}  \|x - y \|_\infty.
\end{eqnarray*}
It completes the proof.
\end{proof}

\subsection{Proof of Lemma \ref{lemma-a-upper-bound}}
\label{section-Lemma5}

\begin{proof}[Proof of Lemma 5]
Recall that $t_i = \sum_{j\neq i} t_{ij}$ and $a_i$ is the number of wins
of subject $i$ out of $t_i$ comparisons. Since all comparisons are mutually independent, $a_i$ is the sum of $m_i$ independent Bernoulli
random variables given $t_{ij} = m_{ij}$, for $j = 0, \ldots, n$, where $m_i = \sum_{j\neq i} m_{ij}$.
By \citeauthor{Hoeffding:1963}'s (\citeyear{Hoeffding:1963}) inequality, we have
\begin{eqnarray*}
&&\P\left(|a_i-\E (a_{ij}|t_{ij}, j=0, \ldots, n )|\ge \sqrt{2m_i \log n}|
t_{ij}=m_{ij}, j=0, \ldots, n \right) \\
&&\P\left(|a_i-\E (a_{ij}|t_{ij}, j=0, \ldots, n )|\ge \sqrt{2m_i \log n}|
t_{ij}=m_{ij}, j=0, \ldots, n \right) \\
& \le & 2\exp{\{-\frac{2m_i\log n }{m_i}\}} =  \frac{2}{n^{2}}.
\end{eqnarray*}
where $\E(a_{ij} | t_{ij}, j=0, \ldots, n)$ is the conditional expectation given $t_{ij}$ for $0\le j \le n$.
Note that the upper bound of the above probability does not
depend on $m_{ij}$. With the law of total probability, for fixed $i$,
\begin{eqnarray*}
&&\P\left(|a_i- \sum_j \E (a_{ij}|t_{ij}, j=0, \ldots, n )|\ge \sqrt{2t_i \log n}|\right) \\
& = &
\sum_{m_{i0}=0}^T \cdots \sum_{m_{in}=0}^T \P( t_{ij}= m_{ij}, j=0, \ldots, n)\\
&&\times \P\left(|a_i- \sum_j\E (a_{ij}|t_{ij}, j=0, \ldots, n )|\ge \sqrt{2m_i \log n}|
t_{ij}=m_{ij}, j=0, \ldots, n \right) \\
&\le & 2n^{-2} \sum_{m_{i0}=0}^T \cdots \sum_{m_{in}=0}^T \P( t_{ij}= m_{ij}, j=0, \ldots, n) \\
& \le & 2\exp{\{-\frac{2m_i\log n }{m_i}\}} =  \frac{2}{n^{2}}.
\end{eqnarray*}
Therefore,
\begin{eqnarray*}
&&\P\left( \max\limits_{i=1, \ldots, n} |a_i- \sum_j \E (a_{ij}|t_{ij}, j=0, \ldots, n ) |\ge   \sqrt{2t_{\max} \log n} \right) \\
& \le & \P \left( \bigcup_{i}\left\{ |a_i-\E a_i|\ge  \sqrt{2t_i \log n}  \right\}\right ) \\
&\le & \sum_{i=1}^{n}\P\left(|a_i-\E a_i|\ge \sqrt{2t_i\log n}\right)\\
&\le & n\times\frac{1}{n^2}=\frac{1}{n}.
\end{eqnarray*}
It completes the proof.
\end{proof}

\subsection{Proof of Lemma \ref{lemma-chernoff-bound-ti}}
\label{section-lemma6}

\begin{proof}[Proof of Lemma \ref{lemma-chernoff-bound-ti}]
We first evaluate the uniform lower bound of $t_i$, $i=0, \ldots, n$.
Note that $t_i$ is the sum of $nT$ independent and identically distributed (i.i.d.)
binomial random variables, $Bin(T, p_n)$. It can be also regarded as the sum of $Tn$, i.i.d., Bernoulli random variables.
With the use of Chernoff bound (\cite{chernoff1952}), we have
\[
\P( \min_{i=0, \ldots, n} t_i < \tfrac{ T}{2} np_n ) \le \sum_{i=0}^n \P( t_i < \tfrac{T}{2}np_n) \le (n+1) \exp( - \tfrac{T}{8}np_n ).
\]
Thus, with probability at least $ 1 - (n+1) \exp( - \tfrac{T}{8}np_n )$,
\[
\min_{i=0, \ldots, n} t_i \ge \tfrac{T}{2}np_n.
\]
Analogously, with the use of Chernoff bound (\cite{chernoff1952}), we have
\[
\P( \max_{i=0, \ldots, n} t_i > \tfrac{ 3}{2} nTp_n ) \le \sum_{i=0}^n \P( t_i > \tfrac{3}{2}nTp_n) \le (n+1) \exp( - \tfrac{1}{10}nTp_n ),
\]
and
\[
\P( \tfrac{ 1}{2}\sum_i t_i > \tfrac{ 3}{2} (n+1 )nTp_n ) \le \exp(- \tfrac{1}{10}(n+1 )nTp_n ).
\]
\end{proof}

\subsection{Proof of Lemma \ref{lemma-cov-wh}}
\label{section-lemma7}

\begin{proof}[Proof of Lemma \ref{lemma-cov-wh}]  
Write
\[
\begin{array}{c}
H=H(\beta^*), ~~V^{-1}=H^\prime(\beta^*), \\
\E^*(\cdot)=\E(\cdot|t_{ij}, 0\le i,j\le n),
~~\mathrm{Cov}^*(\cdot)=\mathrm{Cov}(\cdot|t_{ij}, 0\le i,j\le n).
\end{array}
\]
Then,
\[
H=\E^*(a) - a.
\]
Let $W=V^{-1}-S$. Note that $U=\mathrm{Cov}^*(H)$. By direct calculations, we have
\[
\mathrm{Cov}^*(WH)  = (V^{-1} -S ) U (V^{-1} -S ) := I_1 + I_2
\]
where $I_1=V^{-1} -S  + SVS - S$ and $I_2 = (V^{-1} -S ) (U-V) (V^{-1} -S )$.
It is easy to verify
\[
(SVS-S)_{ij} = \frac{ v_{i0} }{v_{ii}v_{00}} + \frac{ v_{0j}}{v_{jj}v_{00} } - \frac{ ( 1-\delta_{ij})v_{ij}}{v_{ii}v_{jj}}.
\]
Therefore,
\[
\max_{i,j}|(SVS-S)_{ij} | \le \frac{ 3 b_{n1}}{ t_{\min}^2 b_{n0}^2 }.
\]
By Lemmas 3 and 5 in the main text, we have
\[
I_1 = O_p( \frac{ b_{n1}^2}{n^2b_{n0}^3 p_n^5 } ).
\]
Now, we evaluate $I_2$. Direct calculations give
\begin{align*}
    &~~ [(V^{-1} -S ) (U-V) (V^{-1} -S ) ]_{ij} \\
=   &~~ \sum_{k,s}  (V^{-1} -S )_{ik} (U-V)_{ks} (V^{-1} -S ) ]_{sj} \\
=   &~~  O( (\frac{ b_{n1}^2 }{ n^2 b_{n0}^3 p_n \eta_n })^2 ) \sum_{k,s} |(U-V)_{ks}| \\
= &~~ O( (\frac{ b_{n1}^2 }{ n^2 b_{n0}^3 p_n \eta_n })^2 ) \times 2(b_{n1}+1/4)\sum_i t_i \\
=   &~~ O( \frac{ b_{n1}^5 }{ n^2 b_{n0}^6 p_n^5 } ).
\end{align*}
where the second inequality is due to Lemma 3 and the third inequality is due to that $p_{ij}(1-p_{ij}) \le 1/4$ and $\mu^\prime_{ij}(\beta) \le b_{n1}$.
Therefore, we have
\[
I_1 + I_2 = O_p( \frac{ b_{n1}^5 }{ n^2 b_{n0}^6 p_n^5 } ).
\]
\end{proof}

\setlength{\itemsep}{-1.5pt}
\setlength{\bibsep}{0ex}
\bibliography{reference3}
\bibliographystyle{apalike}

\end{document}